\documentclass[12pt]{amsart}
\usepackage{amsmath,amsfonts,amsthm,amssymb,mathrsfs,
amsxtra,amscd,latexsym}
\usepackage[hmargin=2cm,vmargin=3cm]{geometry}
\usepackage[dvipdfm, usenames]{color}
 \newtheorem{thm}{Theorem}[section]

 \newtheorem{lem}[thm]{Lemma}
 \newtheorem{prop}[thm]{Proposition}
 
 \newtheorem{dfn}[thm]{Definition}

 \newtheorem{rmk}[thm]{Remark}
 \theoremstyle{definition}
 \theoremstyle{remark}
 \numberwithin{equation}{section}

\newcommand{\sm}{\left(\begin{smallmatrix}}
\newcommand{\esm}{\end{smallmatrix}\right)}
\newcommand{\mat}{\left(\begin{matrix}}
\newcommand{\emat}{\end{matrix}\right)}

\newcommand{\mb}{\mathbb}
\newcommand{\mbf}{\mathbf}

\newcommand{\mc}{\mathcal}

\def\CC{\mathbb{C}}
\def\HH{\mathbb{H}}

\def\QQ{\mathbb{Q}}
\def\RR{\mathbb{R}}
\def\ZZ{\mathbb{Z}}

\def\det{\mathrm{det}}

\def\Ker{\mathop{\rm Ker}}
\def\m{\mathrm{mod}}

\def\Res{\mathop{\rm Res}}

\def\tr{\mathrm{tr}}

\def\GL{\mathrm{GL}}
\def\Mp{\mathrm{Mp}}
\def\SL{\mathrm{SL}}

\begin{document}

\title{The Eichler cohomology theorem for Jacobi forms}


 \author{Dohoon Choi}
\author{Subong Lim}

 \address{School of liberal arts and sciences, Korea Aerospace University, 200-1, Hwajeon-dong, Goyang, Gyeonggi 412-791, Republic of Korea}
  \email{choija@kau.ac.kr}	

 \address{Korea Institute for Advanced Study(KIAS), 85 Hoegiro (Cheongnyangni-dong 207-43), Dongdaemun-gu, Seoul 130-722, Republic of Korea}
  \email{subong@kias.re.kr}

 \thanks{Keynote:  Jacobi form, Eichler cohomology}

\begin{abstract}
Let $\Gamma$ be a finitely generated Fuchsian group of the first kind which has at least one parabolic class. Eichler introduced a cohomology theory for Fuchsian groups, called as ``Eichler cohomology theory", and established the $\CC$-linear isomorphism from the direct sum of two spaces of cusp forms on $\Gamma$ with the same integral weight to the Eichler cohomology group of $\Gamma$. After the results of Eichler, the Eichler cohomology theory was generalized in various ways. For example, these results were generalized by Knopp to the cases with arbitrary real weights. In this paper, we extend the Eichler cohomology theory to the context of Jacobi forms. We define the cohomology groups of Jacobi groups which are analogues of Eichler cohomology groups and prove an Eichler cohomology theorem for Jacobi forms of arbitrary real weights. Furthermore, we prove that every cocycle is parabolic and that for some special cases we have an isomorphism between the cohomology group and the space of Jacobi forms in terms of the critical values of partial $L$-functions of Jacobi cusp forms.
\end{abstract}

\maketitle

\section{Introduction}
In \cite{E}, Eichler conceived the Eichler cohomology theory while studying generalized abelian integrals, which are now called the Eichler integrals. The Eichler cohomology theory is a cohomology theory for Fuchsian groups. More precisely, if we let $\Gamma$ be a finitely generated Fuchsian group of the first kind which has at least one parabolic class, Eichler defined the first cohomology group of $\Gamma$ with a certain module of polynomials as a coefficient and  established that this group is isomorphic to the direct sum of two spaces of cusp forms on $\Gamma$ with the same integral weight. After the results of Eichler, the Eichler cohomology theory was generalized in various ways.
For example, Gunning \cite{G2} and Kra \cite{K2,K3,K4,K1,K5} gave analogues of Eichler's results for higher dimensional cohomology groups and more general groups $\Gamma$, respectively, by using the spaces of holomorphic functions  on the complex upper half plane $\mathbb{H}$ as $\Gamma$-modules. Furthermore, Knopp \cite{K} introduced an infinite dimensional space of holomorphic functions satisfying a certain growth condition on which $\Gamma$ acts in a manner analogous to the action of $\Gamma$ on vector spaces of polynomials and adjusted the Eichler cohomology theorem to the situation that the weights of cusp forms are arbitrary real.

The extension of the Eichler cohomology theory to  more general functions other than modular forms such as generalized modular forms (GMF) and vector-valued modular forms are considered by many researchers (for example, see \cite{G, KLR, KR, L3, R}). For Jacobi forms, Choie and the second author studied in \cite{CL} the Jacobi integral analogous to the Eichler integral. The authors of \cite{CL} gave examples of Jacobi integrals including generalized Jacobi Poincar\'e series and suggested a definition of the cohomology group of a Jacobi group. In spite of this progress, it is still mysterious if there is an isomorphism between the cohomology group of a Jacobi group and the space of Jacobi forms on the group. Moreover, unfortunately it turns out that the first cohomology groups defined in \cite{CL} are infinite dimensional (see section \ref{section5} for details). So the aim of this paper is to refine the definition of the cohomology group of a Jacobi group given in \cite{CL} and to prove the existence of an isomorphism between the cohomology group of a Jacobi group and the space of Jacobi forms on the group.

A Jacobi form is a function of two variables $\tau\in\HH$ and $z\in\CC$,
appearing in the Fourier expansion of Siegel modular forms of degree 2, which satisfies modular transformation properties with respect to $\tau$, elliptic transformation properties with respect to $z$, and a certain growth condition.
A theory of Jacobi forms was developed systematically by Eichler and Zagier in \cite{EZ}.
In fact Jacobi forms played an important role in the proof of the Saito-Kurokawa conjecture (see \cite{A, EZ, M1, M2, M3,Zag}). After these works the theory of Jacobi forms was extensively developed with beautiful applications in many areas of mathematics and physics. For example, Zwegers in \cite{Z} constructed  special Jacobi forms that are crucial in studying the theory of mock theta functions. Jacobi forms also appeared in many literature on the theory of Donaldson invariants of $\CC\mb{P}^2$ that are related to gauge theory (see for example G$\ddot{\mathrm{o}}$ttsche and Zagier \cite{GZ}), and in recent work on the Mathieu moonshine (for example \cite{EO}). We denote the vector space of Jacobi forms (resp. cusp forms) of weight $k$, index $\mc{M}$ and multiplier system $\chi$ on $\Gamma^{(1,j)}$ by $J_{k,\mc{M},\chi}(\Gamma^{(1,j)})$ (resp. $S_{k,\mc{M},\chi}(\Gamma^{(1,j)})$), where $\Gamma\subset\SL(2,\ZZ)$ is a finitely generated Fuchsian group of the first kind which has at least one parabolic class and  $\Gamma^{(1,j)}$ denotes the {\it Jacobi group} $\Gamma \ltimes (\ZZ^{(j,1)})^2$. Here, for a ring $R$, $R^{(n,j)}$ is the set of $n$ by $j$ matrices whose entries are in $R$.

Now we investigate the existence of an isomorphism between the cohomology group of $\Gamma^{(1,j)}$ and the space of Jacobi forms on $\Gamma^{(1,j)}$. 
For $\gamma=\sm a&b\\c&d\esm \in\Gamma , X = (\lambda, \mu)\in (\ZZ^{(j,1)})^2$ and $\mc{M}\in\ZZ^{(j,j)}$ with $\mc{M}>0$ symmetric, we define
\[(\Phi|_{k,\mc{M},\chi} \gamma)(\tau,z) := (c\tau+d)^{-k}\bar{\chi}(\gamma)e^{-2\pi i\frac{c}{c\tau+d}\tr(\mc{M}zz^t)}\Phi(\gamma(\tau,z))\]
and
\[(\Phi|_{\mc{M}}X)(\tau,z) :=e^{2\pi i\tr(\mc{M}(\lambda\lambda^t\tau+ 2\lambda z^t+\mu\lambda^t))}\Phi(\tau,z+\lambda\tau+\mu),\]
where $\gamma(\tau,z) = (\frac{a\tau+b}{c\tau+d},\frac{z}{c\tau+d})$.
Then $\Gamma^{(1,j)}$ acts on the space of  functions on $\HH\times\CC^{(j,1)}$ by
\begin{equation} \label{slashoperator}
(\Phi|_{k,\mc{M},\chi} (\gamma,X))(\tau,z) := (\Phi|_{k,\mc{M},\chi} \gamma |_{\mc{M}}X)(\tau,z).
\end{equation}


We consider the following special $\mathbb{C}[\Gamma^{(1,j)}]$-module to define a cohomology group of $\Gamma^{(1,j)}$.

 \begin{dfn}\label{dfnofpme}
Let $\mc{P}^e_{\mc{M}}$ be the set of holomorphic functions $g(\tau,z)$ on $\HH\times\CC^{(j,1)}$ which satisfy the following conditions:
\begin{enumerate}
\item[(1)] $|g(\tau,z)| < K(|\tau|^\rho + v^{-\sigma})e^{2\pi\tr(\mc{M}yy^t)/v}$
for some positive constants $K,\rho$ and $\sigma$, where $\tau = u+iv\in\HH$ and $z = x+iy\in\CC^{(j,1)}$,
\item[(2)] $(g|_{\mc{M}} X)(\tau,z) = g(\tau,z)$ for every $X\in (\ZZ^{(j,1)})^2$.
\end{enumerate}
\end{dfn}
This set $\mc{P}^e_{\mc{M}}$ is preserved under the slash operator $|_{k,\mc{M},\chi}(\gamma,X)$ for any $k$, $\chi$ and $(\gamma,X)\in \Gamma^{(1,j)}$ and forms a vector space over $\CC$.
A collection $\{p_{(\gamma,X)}|\ (\gamma,X)\in \Gamma^{(1,j)}\}$ of elements of $\mc{P}^e_{\mc{M}}$ satisfying
\begin{equation} \label{cocycle}
p_{(\gamma_1,X_1)(\gamma_2,X_2)}(\tau,z) = (p_{(\gamma_1,X_1)}|_{-k+\frac j2,\mc{M},\chi}(\gamma_2,X_2))(\tau,z) + p_{(\gamma_2,X_2)}(\tau,z)
\end{equation}
is called a {\it{cocycle}} and a {\it{coboundary}} is a collection $\{p_{(\gamma,X)}|\ (\gamma,X)\in \Gamma^{(1,j)}\}$ such that
\[p_{(\gamma,X)}(\tau,z) = (p|_{-k+\frac j2, \mc{M},\chi}(\gamma,X))(\tau,z) - p(\tau,z),\]
for all $(\gamma,X)\in \Gamma^{(1,j)}$ with a fixed element $p(\tau,z)$ of $\mc{P}^e_{\mc{M}}$. The {\it{cohomology group}} $H^1_{-k+\frac j2,\mc{M},\chi}(\Gamma^{(1,j)},\mc{P}^e_{\mc{M}})$ is the quotient of the cocycles by the coboundaries. A cocycle $\{p_{(\gamma,X)}|\ (\gamma,X)\in \Gamma^{(1,j)}\}$ is called the {\it{parabolic cocycle}} if for every parabolic element $B$ in $\Gamma$ there exists a function $Q_{B}(\tau,z)$ in $\mc{P}^e_{\mc{M}}$ such that
\[p_{(B,0)}(\tau,z) = (Q_{B}|_{-k+\frac j2,\mc{M},\chi})(B,0)(\tau,z) - Q_{B}(\tau,z).\]
 The {\it{parabolic cohomology group}} $\tilde{H}^1_{-k+\frac j2,\mc{M},\chi}(\Gamma^{(1,j)},\mc{P}^e_{\mc{M}})$ is defined as the vector space obtained by forming the quotient of the parabolic cocycles by the coboundaries. In the following theorem
we prove the existence of an isomorphism between the parabolic cohomology group of $\Gamma^{(1,j)}$ and the space of Jacobi cusp forms on $\Gamma^{(1,j)}$.

\begin{thm} \label{main1} For a real number $k$, index $\mc{M}\in \ZZ^{(j,j)}$ with $\mc{M}>0$ and multiplier system $\chi$ of weight $-k+\frac j2$, we have an isomorphism
\[\tilde{\eta}: S_{k+2+\frac j2, \mc{M},\bar{\chi}}(\Gamma^{(1,j)}) \cong \tilde{H}^1_{-k+\frac j2,\mc{M},\chi}(\Gamma^{(1,j)},\mc{P}^e_{\mc{M}}).\]
\end{thm}

For a cohomology group $H^1_{-k+\frac j2,\mc{M},\chi}(\Gamma^{(1,j)},\mc{P}^e_{\mc{M}})$ without the parabolic condition, we prove the following isomorphism, which comes from the fact that every cocycle is parabolic.

\begin{thm} \label{main2} Under the same map as in Theorem \ref{main1}, we have an isomorphism
\[\tilde{\eta}: S_{k+2+\frac j2, \mc{M},\bar{\chi}}(\Gamma^{(1,j)}) \cong H^1_{-k+\frac j2,\mc{M},\chi}(\Gamma^{(1,j)},\mc{P}^e_{\mc{M}}).\]
\end{thm}


For some special cases we have an explicit isomorphism
\[\tilde{\eta}:S_{k+2+\frac j2, \mc{M},\chi}(\Gamma^{(1,j)}) \to \tilde{H}^1_{-k+\frac j2,\mc{M},\chi}(\Gamma^{(1,j)},\mc{P}^e_{\mc{M}})\]
 in term of the critical values of partial $L$-functions of Jacobi cusp forms.More precisely, suppose that $j=1$ and $k\in\ZZ_{>0}$ and that $\Phi(\tau,z)\in S_{k+2+\frac12,m,\chi}(\Gamma^{(1,1)})$, where $m\in\ZZ_{>0}$.
Let $T = \sm 1&\lambda\\ 0&1\esm,\ \lambda>0$, be a generator of $\Gamma_\infty$ and $\chi(T) = e^{2\pi i\kappa}$. Then
\[\Phi(\tau,z) = \sum_{4m(n+\kappa)-\lambda r^2>0} c\biggl(\frac{n+\kappa}{\lambda},r\biggr)q^{(n+\kappa)/\lambda}\zeta^r,\]
where $q = e^{2\pi i\tau},\ \zeta = e^{2\pi iz}$.
Let $N(n) = \frac{4m(n+\kappa)}{\lambda}-r^2$. Then $\Phi(\tau,z)$ also can be written as
\[\Phi(\tau,z) = \sum_{\mu (\m\ 2m)} \sum_{r\equiv \mu (\m 2m)} \sum_{n\in\ZZ\atop N(n)>0} C_\mu(N(n))q^{\frac{N(n)+r^2}{4m}}\zeta^r,\]
where $C_\mu(N) := c\biggl(\frac{N(n)+r^2}{4m},r\biggr)$. Then the {\it{partial\ L-functions}} of $\Phi(\tau,z)$ are defined by
\[L(\Phi,\mu,\gamma,s) = \sum_{n\in\ZZ\atop N(n)>0} \frac{C_\mu(N(n)) e^{2\pi i\frac{-dN(n)}{4mc}}} {(N(n)/4m)^s},\]
for $\gamma = \sm a&b\\ c&d\esm\in\Gamma$.
Then $\tilde{\eta}(\Phi)$ is a cocycle class given by a cocycle representative in the following way.
\begin{thm} \label{main3}
Suppose that a function $r(\Phi,(\gamma,X);\tau,z)$ is given by
\[r(\Phi,(\gamma,X); \tau,z) = \sum_{\mu=1}^{2m} \sum_{r\equiv\mu(\m\ 2m)} \sum_{k=0}^{n} \frac{k!(-1)^{k+n} \overline{L(\Phi, \mu,\gamma,n+1)}}{(k-n)!(2\pi i)^{n+1}}\biggl(\tau+\frac dc\biggr)^{k-n}q^{\frac{r^2}{4m}}\zeta^r,\]
where $(\gamma,X) \in\Gamma^{(1,1)}$ and $\gamma = \sm a&b\\c&d\esm$. Then a collection $\{r(\Phi,(\gamma,X);\tau,z)|\ (\gamma,X)\in\Gamma^{(1,1)}\}$ is a cocycle representative of $\tilde{\eta}(\Phi)$.
\end{thm}

\begin{rmk} It is well known that if $f(\tau)$ is a modular form of weight $k+2$ then its period functions  are polynomials of degree at most $k$ and their coefficients  can be written in terms of the critical values of $L$-functions associated with a given modular form $f(\tau)$. The periods of modular forms are rich source of arithmetic of modular forms (see \cite{KZ}).
\end{rmk}

To prove our main theorems  we extend the argument of Knopp in \cite{K} for modular forms to Jacobi forms. Note that a Jacobi form is defined not on a complex curve but a complex surface. So one of main features of our proofs is overcoming difficulties coming from the complexities of the structure of a complex surface by the theta decomposition theory, developed by Eichler and Zagier \cite{EZ}, Ziegler \cite{Z0}, which gives an isomorphism between Jacobi forms and vector-valued modular forms.

The remainder of this paper is organized as follows. In section \ref{section2}, we review the theory of
Jacobi forms. In section \ref{section3},
we introduce vector-valued modular forms and establish the Eichler cohomology theory for vector-valued modular
forms of real weights. More precisely, we describe a map from the space of vector-valued cusp forms of a
real weight to the cohomology group of $\Gamma$ with a certain module of vector-valued functions and
prove that this map is an isomorphism. In section \ref{section4}, we prove the main theorems:
Theorem \ref{main1}, \ref{main2} and  \ref{main3}. Finally, in section \ref{section5}, we give some remarks regarding the cohomology group defined in \cite{CL} by the comparision with our cohomology group.

\bigskip

\section{Jacobi forms} \label{section2}

 Eichler and Zagier laid the foundations for the theory of Jacobi forms in \cite{EZ} and Ziegler \cite{Z0} investigated higher dimensional cases in the spirit of Eichler and Zagier.
In this section,  we review basic notions of Jacobi forms (see \cite{EZ, Z0} for details).
First we fix some notations.
 Let $\Gamma \subset \SL(2,\ZZ)$ be a finitely generated Fuchsian group of the first kind, which has at least one parabolic class and $\Gamma^{(1,j)}=\Gamma\ltimes (\ZZ^{(j,1)})^2$ be a {\it{Jacobi group}} with associated composition law:
\[(\gamma_1,(\lambda_1,\mu_1))\cdot(\gamma_2,(\lambda_2,\mu_2))=(\gamma_1 \gamma_2, (\widetilde{\lambda_1}+\lambda_2,\widetilde{\mu_1}+\mu_2)),\]
where $(\widetilde{\lambda},\widetilde{\mu})=(\lambda,\mu)\cdot \gamma_2$.
Then $\Gamma^{(1,j)}$ acts on $\HH\times\CC^{(j,1)}$ as a group of automorphism. The action is given by:
\[(\gamma, (\lambda,\mu))\cdot(\tau,z)= \biggl(\gamma\tau,\frac{z+\lambda\tau+\mu}{c\tau+d}\biggr),\]
where $\gamma\tau = \frac{a\tau+b}{c\tau+d}$ for $\gamma = \sm a&b\\c&d\esm\in\Gamma$.
Let $k$ be a real number and $\chi$ be a multiplier system of weight $k$ on $\Gamma$, i.e., $\chi : \Gamma\to\CC$ satisfies
\begin{enumerate}
\item[(1)] $|\chi(\gamma)| = 1$ for all $\gamma\in\Gamma$,
\item[(2)] $\chi$ satisfies the consistency condition
\[\chi(\gamma_3)(c_3\tau+d_3)^{k}=\chi(\gamma_1)\chi(\gamma_2)(c_1\gamma_2\tau+d_1)^{k}(c_2\tau+d_2)^{k},\]
where $\gamma_3=\gamma_1\gamma_2$ and $\gamma_i =\sm *&*\\c_i&d_i\esm, i =1,2$ and $3$,
\item[(3)] $\chi$ satisfies the nontriviality condition
\[\chi(-I) = e^{\pi ik}.\]
\end{enumerate}

With these notations, we introduce the definition of a Jacobi form.

\begin{dfn}
A {\it{Jacobi form}} of weight $k$, index $\mc{M}$ and multiplier system $\chi$ on $\Gamma^{(1,j)}$  is a holomorphic mapping $\Phi(\tau,z)$ on $\HH\times\CC^{(j,1)}$ satisfying
\begin{enumerate}
\item[(1)] $(\Phi|_{k,\mc{M},\chi} \gamma)(\tau,z) =\Phi(\tau,z)$ for every $\gamma\in\Gamma$,
\item[(2)] $(\Phi|_{\mc{M}}X)(\tau,z) = \Phi(\tau,z)$ for every $X\in (\ZZ^{(j,1)})^2$,
\item[(3)] for each $\gamma=\sm a&b\\c&d\esm \in\SL(2,\ZZ)$, the function $(c\tau+d)^{-k}e^{2\pi i\tr(\mc{M}zz^t)\frac{-c}{c\tau+d}}\Phi((\gamma,0)\cdot(\tau,z))$ has the Fourier expansion of the form
\begin{equation} \label{Jacobifourier}
(c\tau+d)^{-k}e^{2\pi i\tr(\mc{M}zz^t)\frac{-c}{c\tau+d}}\Phi((\gamma,0)\cdot(\tau,z)) =
\sum_{l\in\ZZ\atop l+\kappa_\gamma\geq0}\sum_{r\in\ZZ^{(1,j)}}a(l,r)e^{2\pi i\tau(l+\kappa_{\gamma})/\lambda_\gamma}e^{2\pi i\tr(rz)},
\end{equation}
with a suitable $0\leq \kappa_\gamma<1,\ \lambda_\gamma\in\ZZ$ and $a(l,r)\neq0$ only if $4(l+\kappa_\gamma)-\lambda_\gamma r\mc{M}^{-1}r^t\geq0$.
\end{enumerate}
\end{dfn}

We denote by $J_{k,\mc{M},\chi}(\Gamma^{(1,j)})$ the vector space of all Jacobi forms of weight $k$, index $\mc{M}$ and multiplier system $\chi$ on $\Gamma^{(1,j)}$. If a Jacobi form satisfies the  condition $a(l,r)\neq 0$ only if $4(l+\kappa_\gamma)-\lambda_\gamma r\mc{M}^{-1}r^t>0$, then it is called a {\it{Jacobi cusp form}}. We denote by $S_{k,\mc{M},\chi}(\Gamma^{(1,j)})$ the vector space of all Jacobi cusp forms of weight $k$, index $\mc{M}$ and multiplier system $\chi$ on $\Gamma^{(1,j)}$.

Now we look into the theta series, which plays an important role in the proofs of our main theorems.
Let $S\in\ZZ^{(j,j)}$ be symmetric, positive definite and $a,b\in\QQ^{(j,1)}$. We consider the theta series
\[\theta_{S,a,b}(\tau,z) := \sum_{\lambda\in\ZZ^{(j,1)}}e^{\pi i\tr(S((\lambda+a)(\lambda+a)^t\tau+2(\lambda+a)(z+b)^t))}\]
with characteristic $(a,b)$ converging normally on $\HH\times\CC^{(j,1)}$. Then this theta series satisfies the following transformation properties.

\begin{lem}  \label{lem1} \cite[Section 5]{EZ}, \cite[Lemma 3.2]{Z0}
Let $\mathcal{N}$ be a complete system of representatives of the cosets \[(2\mc{M})^{-1}\ZZ^{(j,1)}/\ZZ^{(j,1)}.\]
For $a\in\mathcal{N}$ we have
\begin{enumerate}
\item[(1)] $\theta_{2\mc{M},a,0}\biggl(-\frac{1}{\tau}, \frac{z}{\tau}\biggr) = \det(2\mc{M})^{-\frac 12}\det\biggl(\frac \tau i\biggr)^{\frac j2}e^{2\pi i\tr(\mc{M}zz^t/\tau)}\sum_{b\in\mathcal{N}}e^{-2\pi i\tr(2\mc{M}ba^t)}\theta_{2\mc{M},b,0}(\tau,z)$,
\item[(2)] $\theta_{2\mc{M},a,0}(\tau+1,z) = e^{2\pi i\tr(\mc{M}aa^t)}\theta_{2\mc{M},a,0}(\tau,z)$.
\end{enumerate}
\end{lem}

We are in a position to explain the theta expansion. To do that, we need to introduce the following space of functions on $\CC^{(j,1)}$.
For $\tau_0\in\HH$, let $T_{\mc{M}}(\tau_0)$ denote the vector space of all holomorphic functions $g:\CC^{(j,1)}\to\CC$ satisfying
\[g(z+\lambda \tau_0+\mu) = e^{-2\pi i\tr(\mc{M}(\lambda \lambda^t\tau_0+2\lambda z^t))}g(z)\]
for every $\lambda, \mu\in\ZZ^{(j,1)}$. Writing this functional equation in terms of Fourier coefficients yields the following lemma.

\begin{lem} \label{lem2} \cite[Lemma 3.1]{Z0}
 The functions $\{\theta_{2\mc{M},a,0}(\tau_0,z)|\ a\in\mathcal{N}\}$ form a basis of $T_{\mc{M}}(\tau_0)$. Especially we have
 \[\dim_{\CC}T_{\mc{M}}(\tau_0) = \det(2\mc{M}).\]
\end{lem}

Using Lemma \ref{lem1} and Lemma \ref{lem2}, we prove an isomorphism between $J_{k,\mc{M},\chi}(\Gamma^{(1,j)})$ and a certain space of vector-valued  modular forms from the following result.

\begin{thm}  \cite[Section 5]{EZ}, \cite[Section 3]{Z0} \label{decomposition} Let $\Phi(\tau,z)$ be holomorphic as a function of $z$ and satisfy
\begin{equation} \label{elliptic}
(\Phi|_{\mc{M}} X)(\tau,z) = \Phi(\tau,z)\ \text{for every}\ X\in (\ZZ^{(1,j)})^2.
\end{equation}
Then we have
\begin{equation} \label{thetaexpansion}
\Phi(\tau,z) = \sum_{a\in\mathcal{N}}f_a(\tau)\theta_{2\mc{M},a,0}(\tau,z)
\end{equation}
with uniquely determined holomorphic functions $f_a:\HH\to\CC$.
If $\Phi(\tau,z)$ also satisfies the transformation
\[(\Phi|_{k,\mc{M},\chi} \gamma)(\tau,z) =\Phi(\tau,z)\ \text{for every}\ \gamma\in\SL(2,\ZZ),\]
then we have for each $a\in\mc{N}$
\[f_a\biggl(-\frac1\tau\biggr) = \chi(\sm 0&-1\\1&0\esm)\det\biggl(\frac \tau i\biggr)^{-\frac j2}\tau^k\det(2\mc{M})^{-\frac 12}\sum_{b\in\mathcal{N}}e^{2\pi i\tr(2\mc{M}ab^t)}f_b(\tau)\]
and
\[f_a(\tau+1) = \chi(\sm 1&1\\0&1\esm)e^{-2\pi i\tr(\mc{M}aa^t)}f_a(\tau).\]
Furthermore, if $\Phi(\tau,z)$ is a Jacobi form in $J_{k,\mc{M},\chi}(\Gamma^{(1,j)})$, then
functions in $\{f_a|\ a\in\mathcal{N}\}$ necessarily must have the Fourier expansions of the form
\[f_a(\tau) =\sum_{l\geq0\atop \text{rational}} a(l)e^{2\pi il\tau}.\]
\end{thm}

The decomposition by theta functions as in (\ref{thetaexpansion}) is called the {\it{theta expansion}}.
Now we explain the isomorphism between Jacobi forms and vector-valued modular forms induced by the theta expansion more precisely.
First, we look into a special representation to define the space of vector-valued modular forms, which corresponds to $J_{k,\mc{M},\chi}(\Gamma^{(1,j)})$.
From now,
we use the notation $\tau = u+iv\in\HH$ and $z = x+iy\in\CC^{(j,1)}$.
 Let $\mbf{e}_a (a\in\mc{N})$ be the standard basis of $\CC^{|\mc{N}|}$. If $j$ is an even integer, then we define a representation from $\SL
(2,\ZZ)$ to $\GL(|\mc{N}|,\CC)$ by
\begin{equation} \label{representation1}
\rho(T)\mbf{e}_a = e^{-2\pi i\tr(\mc{M}aa^t)}\mbf{e}_a
\end{equation}
and
\begin{equation} \label{representation2}
\rho(S)\mbf{e}_a = \frac{i^{\frac j2}}{\sqrt{\det(2\mc{M})}}\sum_{b\in\mc{N}}e^{2\pi i\tr(2\mc{M}ab^t)}\mbf{e}_b,
\end{equation}
where $T = \sm 1&1\\0&1\esm$ and $S = \sm 0&-1\\1&0\esm$.
But if $j$ is an odd integer, then  $\rho$ is not a representation of $\SL(2,\ZZ)$ because of the ambiguities arising from the choice of square-root. To get a homomorphism one must replace $\SL(2,\ZZ)$ by a double cover. We write $\Mp(2,\RR)$ for the {\it{metaplectic group}}, i.e. the double cover of $\SL(2,\RR)$, realized by the two choices of holomorphic square roots of $\tau\mapsto c\tau+d$ for $\sm a&b\\c&d\esm\in\SL(2,\RR)$. Thus the elements of $\Mp(2,\RR)$ are pairs
\[(\gamma,\phi(\tau)),\]
where $\gamma = \sm a&b\\c&d\esm\in\SL(2,\RR)$, and $\phi$ denotes a holomorphic function on $\HH$ with
\[\phi(\tau)^2= c\tau+d.\]
The product of $(\gamma_1,\phi_1(\tau)), (\gamma_2, \phi_2(\tau))\in\Mp(2,\RR)$ is given by
\[(\gamma_1,\phi_1(\tau)) (\gamma_2, \phi_2(\tau)) = (\gamma_1\gamma_2, \phi_1(\gamma_2\tau)\phi_2(\tau)).\]
The map
\[\sm a&b\\c&d\esm  \mapsto \widetilde{\sm a&b\\c&d\esm} = (\sm a&b\\c&d\esm, \sqrt{c\tau+d})\]
defines a locally isomorphic embedding of $\SL(2,\RR)$ into $\Mp(2,\RR)$. Let $\Mp(2,\ZZ)$ be the inverse image of $\SL(2,\ZZ)$ under the covering map $\Mp(2,\RR)\to\SL(2,\RR)$. It is well known that $\Mp(2,\ZZ)$ is generated by $\widetilde{T}$ and $\widetilde{S}$. We define a representation $\widetilde{\rho}$ of $\Mp(2,\ZZ)$ by $\widetilde{\rho}(\widetilde{T}) = \rho(T)\ \text{and}\ \widetilde{\rho}(\widetilde{S})= \rho(S)$. Then $\tilde{\rho}$ is essentially a Weil representation of $\Mp_2(\ZZ)$ and one can check that this representation is unitary.
We take a multiplier system $\chi'$ of weight $j/2$, for example we can take a power of eta-multiplier system: $\chi'(\gamma) = \biggl(\frac{\eta(\gamma\tau)}{\eta(\tau)}\biggr)^{j}$ for $\gamma\in\SL(2,\ZZ)$, where $\eta(\tau) = e^{\frac{\pi i\tau}{12}}\prod_{n=1}^\infty (1-q^n)$ is the Dedekind eta function.
Then we define a map $\rho' : \SL(2,\ZZ) \to \GL(|\mc{N}|,\CC)$ by
\[\rho'(\gamma) = \chi'(\gamma)\widetilde{\rho}(\widetilde{\gamma})\]
for $\gamma\in\SL(2,\ZZ)$.

\begin{lem} \label{homomorphism} The map $\rho'$ gives a representation of $\SL(2,\ZZ)$ if $j$ is an odd integer.
\end{lem}

\begin{proof} [\bf Proof of Lemma \ref{homomorphism}]
We have to show that $\rho'$ is a homomorphism. For $\gamma_1 = \sm a_1&b_1\\c_1&d_1\esm, \gamma_2=\sm a_2&b_2\\c_2&d_2\esm \in\SL(2,\ZZ)$, we see that by definition
\begin{eqnarray*}
 \rho'(\gamma_1)\rho'(\gamma_2) &=& \chi'(\gamma_1)\widetilde{\rho}(\widetilde{\gamma_1})\chi'(\gamma_2)\widetilde{\rho}(\widetilde{\gamma_2})\\
 &=& \chi'(\gamma_1)\chi'(\gamma_2)\widetilde{\rho}((\gamma_1,\sqrt{c_1\tau+d_1}))\widetilde{\rho}((\gamma_2,\sqrt{c_2\tau+d_2}))\\
&=& \chi'(\gamma_1)\chi'(\gamma_2)\widetilde{\rho}((\gamma_1\gamma_2, \sqrt{c_1\gamma_2\tau+d_1}\sqrt{c_2\tau+d_2})).
\end{eqnarray*}
Since $\chi'$ is a multiplier system of weight $j/2$ and $j$ is an odd integer, $\chi'$ is also a multiplier system of weight $1/2$. So, if we let $\gamma_3 = \gamma_1\gamma_2 = \sm a_3&b_3\\ c_3&d_3\esm$, then we have
\[\chi'(\gamma_3)\sqrt{c_3\tau+d_3} = \chi'(\gamma_1)\chi'(\gamma_2)\sqrt{c_1\gamma_2\tau+d_1}\sqrt{c_2\tau+d_2}.\]
Note that
\[\biggl(\frac{\chi'(\gamma_1)\chi'(\gamma_2)}{\chi'(\gamma_3)}\biggr)^2 = \frac{(c_3\tau+d_3)}{(c_1\gamma_2\tau+d_1)(c_2\tau+d_2)} = 1.\]
Therefore, one can see that $\biggl(I, \frac{\chi'(\gamma_1)\chi'(\gamma_2)}{\chi'(\gamma_3)}\biggr)\in \Mp_2(\ZZ)$. By the law of composition in $\Mp_2(\ZZ)$, we see that
\begin{eqnarray*}
\tilde{\rho}((\gamma_3,\sqrt{c_3\tau+d_3}))\tilde{\rho}\biggl(\biggl(I, \frac{\chi'(\gamma_1)\chi'(\gamma_2)}{\chi'(\gamma_3)}\biggr)\biggr) &=& \tilde{\rho}\biggl(\biggl(\gamma_3, \frac{\chi'(\gamma_3)\sqrt{c_3\tau+d_3}}{\chi'(\gamma_1)\chi'(\gamma_2)}\biggr)\biggr)\\
&=& \tilde{\rho}((\gamma_1\gamma_2,\sqrt{c_1\gamma_2\tau+d_1}\sqrt{c_2\tau+d_2})).
\end{eqnarray*}
So we can check that
\[\rho'(\gamma_1)\rho'(\gamma_2) = \chi'(\gamma_1)\chi'(\gamma_2)\tilde{\rho}((\gamma_3,\sqrt{c_3\tau+d_3}))\tilde{\rho}\biggl(\biggl(I, \frac{\chi'(\gamma_1)\chi'(\gamma_2)}{\chi'(\gamma_3)}\biggr)\biggr).\]
Therefore, it suffices to show that $\chi'(\gamma_1)\chi'(\gamma_2)\tilde{\rho}\biggl(\biggl(I, \frac{\chi'(\gamma_1)\chi'(\gamma_2)}{\chi'(\gamma_3)}\biggr)\biggr) = \chi'(\gamma_3)$.

One has the relations $S^2 = (ST)^3 = Z$, where $Z = \biggl(\sm -1&0\\0&-1\esm, i\biggr)$ is the standard generator of the center of $\Mp_2(\ZZ)$. It is well known that
\[\tilde{\rho}(Z)\mbf{e}_{a} = i^j\mbf{e}_{-a}.\]
From this, we see that
\[\tilde{\rho}((I, -1))\mbf{e}_{a} = \tilde{\rho}(Z^2)\mbf{e}_{a} = (-1)\mbf{e}_{a}.\]
Since $\frac{\chi'(\gamma_1)\chi'(\gamma_2)}{\chi'(\gamma_3)} = \pm1$, it follows that
\[\tilde{\rho}\biggl(\biggl(I, \frac{\chi'(\gamma_1)\chi'(\gamma_2)}{\chi'(\gamma_3)}\biggr)\biggr) = \frac{\chi'(\gamma_3)}{\chi'(\gamma_1)\chi'(\gamma_2)}.\]
Hence, we obtain that
\begin{eqnarray*}
\rho'(\gamma_1)\rho'(\gamma_2) &=& \chi'(\gamma_3)\tilde{\rho}((\gamma_3,\sqrt{c_3\tau+d_3})) = \rho'(\gamma_3) = \rho'(\gamma_1\gamma_2).
\end{eqnarray*}
This completes the proof.
\end{proof}

\medskip

In conclusion, we can define a representation $\rho''$ of $\SL(2,\ZZ)$ induced from the theta expansion as follows
\begin{equation} \label{doublerho}
\rho'' =
\begin{cases}
\rho & \text{if $j$ is an even integer,}\\
\rho' & \text{if $j$ is an odd integer.}
\end{cases}
\end{equation}

Using this we can define the space of vector-valued modular forms associated with $J_{k,\mc{M},\chi}(\Gamma^{(1,j)})$ by the theta expansion.
Let $\Phi(\tau,z)$ be a Jacobi form in $J_{k,\mc{M},\chi}(\Gamma^{(1,j)})$. By Theorem \ref{decomposition}, we have the theta expansion
\[\Phi(\tau,z) = \sum_{a\in\mc{N}}f_a(\tau)\theta_{2\mc{M},a,0}(\tau,z).\]
Then a vector-valued function $\sum_{a\in\mc{N}} f_a(\tau)\mbf{e}_a$ is a vector-valued modular form in  $M_{k-\frac j2, \chi'', \rho''}(\Gamma)$ (for the precise definition of $M_{k-\frac j2, \chi'', \rho''}(\Gamma)$ see section \ref{section3.1}), where
\begin{equation} \label{doublechi}
\chi'' =
\begin{cases}
\chi & \text{if $j$ is an even integer},\\
\chi\overline{\chi'} & \text{if $j$ is an odd integer}.
\end{cases}
\end{equation}
One can see that $\chi''(\gamma)\rho''(\gamma),\ \gamma\in\Gamma,$ is determined independently of the choice of $\chi'$ and hence $M_{k-\frac j2, \chi'', \rho''}(\Gamma)$ is determined uniquely. The important result of the theta expansion is that the map from $J_{k,\mc{M},\chi}(\Gamma^{(1,j)})$ to $M_{k-\frac j2, \chi'',\rho''}(\Gamma)$ given by $\Phi(\tau,z) \mapsto \sum_{a\in\mc{N}}f_a(\tau)\mbf{e}_a$ is actually an isomorphism.

\begin{thm}  \cite[Theorem 5.1]{EZ}, \cite[Theorem 3.3]{Z0} \label{isomorphism}
The theta expansion  gives an isomorphism between $J_{k,\mc{M},\chi}(\Gamma^{(1,j)})$ and $M_{k-\frac j2, \chi'', \rho''}(\Gamma)$.
\end{thm}


\medskip

\bigskip

\section{Vector-valued Modular Forms} \label{section3}

In this section, we prove the Eichler cohomology theorem for vector-valued modular forms.
To do that, we construct a holomorphic vector-valued Eichler integral for a given parabolic cocycle such that its periods are the same as the given parabolic cocycle using a vector-valued generalized Poincar\'e series. Our proof is based on \cite{G, KM2}.

\subsection{Vector-valued modular forms} \label{section3.1}
We begin by introducing the definition of the vector-valued modular forms.
 Let $\Gamma\subseteq\SL(2,\ZZ)$ be a $H$-$\it{group}$, i.e., a finitely generated Fuchsian group of the first kind which has at least one parabolic class. Let $k\in\RR$ and $\chi$ a (unitary) multiplier system of weight $k$ on $\Gamma$.
Let $p$ be a positive integer and $\rho:\Gamma\to \GL(p,\CC)$ a $p$-dimensional unitary complex representation.  We denote the standard basis elements of the vector space $\CC^p$ by $\mbf{e}_j$ for $1\leq j\leq p$. With these setups, the definition of the vector-valued modular forms are given as follows.

\begin{dfn} \label{dfnofvvm} A vector-valued weakly holomorphic modular form of weight $k$, multiplier system $\chi$ and type $\rho$ on $\Gamma$ is a sum $f(\tau) = \sum_{j=1}^p f_j(\tau)\mbf{e}_j$ of functions holomorphic in  $\HH$ satisfying the following conditions:
\begin{enumerate}
\item for all $\gamma = \sm a&b\\c&d\esm\in\Gamma$, we have
$(f|_{k,\chi,\rho} \gamma)(\tau) = f(\tau)$,

\item for each $\gamma=\sm a&b\\c&d\esm\in\SL(2,\ZZ)$,
the function $(c\tau+d)^{-k}f(\gamma\tau)$ has the Fourier expansion of the form
\[(c\tau+d)^{-k}f(\gamma\tau) = \sum_{j=1}^p\sum_{n\gg-\infty} a_{j,\gamma}(n)e^{2\pi i(n+\kappa_{j,\gamma})\tau/\lambda_{\gamma}}\mbf{e}_j,\]
where $\kappa_{j,\gamma}$ (resp. $\lambda_{\gamma}$) is a constant which depends on $j$ and $\gamma$ (resp. $\gamma$).
\end{enumerate}
\end{dfn}

Here, the slash operator $|_{k,\chi,\rho} \gamma$ is defined by
\[(f|_{k,\chi,\rho}\gamma)(\tau) = \chi(\gamma)^{-1}(c\tau+d)^{-k}\rho^{-1}(\gamma)f(\gamma\tau),\]
for $\gamma = \sm a&b\\c&d\esm \in \Gamma$, where $\gamma\tau = \frac{a\tau+b}{c\tau+d}$.
The space of all vector-valued weakly holomorphic modular forms $f(\tau)$ of weight $k$, multiplier system $\chi$ and type $\rho$ on $\Gamma$ is denoted by $M^!_{k,\chi,\rho}(\Gamma)$. There are subspaces $M_{k,\chi,\rho}(\Gamma)$ and $S_{k,\chi,\rho}(\Gamma)$ of {\it{vector-valued\ holomorphic\ modular\ forms}} and {\it{vector-valued cusp forms}}, respectively, for which we require that each $a_{j,\gamma}(n) = 0$ when $n+\kappa_{j,\gamma}$ is negative, respectively, non-positive.

\medskip

\subsection{The Eichler cohomology for vector-valued modular forms} \label{section3.2}
 In this subsection, we define the Eichler cohomology group for a vector-valued modular form of real weight. To define the Eichler cohomology group, first we introduce the coefficient module and the vector-valued Eichler integrals.
 Let $\mc{P}$ be the vector space of vector-valued functions $G(\tau) = \sum_{j=1}^p G_j(\tau)\mbf{e}_j$ holomorphic in $\HH$ which satisfy the growth condition
 \begin{equation} \label{growthconditionforp}
|G_j(\tau)| < K(|\tau|^\rho + v^{-\sigma}),\ v= \Im\tau>0,\ 1\leq j\leq p,
\end{equation}
for some positive constants $K,\rho$ and $\sigma$. Since the weight here is not necessarily in $\ZZ$, polynomials of fixed degree cannot serve as the underlying space of functions in the definition of the cohomology groups in our case. Instead, we employ as the underlying space the collection $\mc{P}$. This space was introduced in \cite{K} in the context of the Eichler cohomology theory for scalar-valued (i.e. the usual) modular forms of arbitrary real weights. It is worth mentioning that the space $\mc{P}$ is preserved under the slash operator. With this underlying space, we define the vector-valued Eichler integrals as follows.

\begin{dfn} Let $\rho$ be a $p$-dimensional representation $\rho: \Gamma\to\GL(p,\CC)$, $k$ an arbitrary real number and $\chi$ a multiplier system for $\Gamma$ of weight $-k$.
A vector-valued Eichler integral of weight $-k$, multiplier system $\chi$ and type $\rho$ on $\Gamma$ is a vector-valued function $F(\tau)$ on $\HH$ satisfying
\[(F|_{-k,\chi,\rho}\gamma)(\tau) - F(\tau)\in \mc{P}\]
for all $\gamma\in\Gamma$.
\end{dfn}

If we let $p_\gamma(\tau) = (F|_{-k,\chi,\rho}\gamma)(\tau) - F(\tau)$, then the vector-valued functions $p_\gamma(\tau)$ are called {\it{period functions}} of $F(\tau)$.
Then it turns out that $\{p_\gamma|\ \gamma\in\Gamma\}$ satisfies the following cocycle condition:
For $\gamma_1,\gamma_2\in\Gamma$,
\begin{equation} \label{cocyclevvm}
p_{\gamma_1\gamma_2}(\tau) = p_{\gamma_2}(\tau) + (p_{\gamma_1}|_{-k,\chi,\rho}\gamma_2)(\tau).
\end{equation}
A collection $\{p_\gamma |\ \gamma\in\Gamma\}$ of elements of $\mc{P}$ satisfying (\ref{cocyclevvm}) is called a {\it{cocycle}} and a {\it{coboudary}} is a collection $\{p_\gamma |\ \gamma\in\Gamma\}$ such that
\[p_\gamma(\tau) = (p|_{-k,\chi,\rho}\gamma)(\tau) - p(\tau)\]
for $\gamma\in\Gamma$ with $p(\tau)$ a fixed element of $\mc{P}$. We define the {\it{cohomology group}} $H^1_{-k,\chi,\rho}(\Gamma,\mc{P})$ as the quotient of the cocycles by the coboundaries. A {\it{parabolic cocycle}} $\{p_\gamma|\ \gamma\in\Gamma\}$ is a collection of elements of $\mc{P}$ satisfying (\ref{cocyclevvm}), in which for every parabolic class $\mc{B}$ in $\Gamma$ there exists a fixed element $Q_B(\tau)\in\mc{P}$ such that
\[p_B(\tau) = (Q_B|_{-k,\chi,\rho}B)(\tau) - Q_B(\tau),\]
for all $B\in\mc{B}$. Note that coboundaries are parabolic cocycles. The {\it{parabolic cohomology group}} $\tilde{H}^1_{-k,\chi,\rho}(\Gamma,\mc{P})$ is defined as the vector space obtained by forming the quotient of the parabolic cocycles by the coboundaries.

\subsection{Vector-valued generalized Poincar\'e series} \label{section3.3}
A vector-valued generalized Poincar\'e series gives an example of the vector-valued Eichler integrals and it is also used in the proof of the Eichler cohomology theorem for vector-valued modular forms.
A vector-valued generalized Poincar\'e series was first defined by Lehner \cite{L3}.
 We recall the definition of a vector-valued generalized Poincar\'e series and investigate its convergence.
 Let $\{g_\gamma|\ \gamma\in\Gamma\}$ be a parabolic cocycle of elements of $\mc{P}$ of weight $-k\in\RR$ with $\chi$ a multiplier system in $\Gamma$ of weight $-k$ and $\rho$ a  representation. Assume also that $g_{Q}(\tau) = 0$, where
\begin{equation} \label{lambda}
Q = \sm 1&\lambda\\ 0&1\esm\in\Gamma,\ \lambda>0,
\end{equation}
 is a generator of $\Gamma_\infty$. We define a {\it{vector-valued generalized Poincar\'e series}} as
 \begin{equation} \label{genpoin}
\Phi(\tau;r) = \sum_{j=1}^p \sum_{V= \sm a&b\\c&d\esm\in\mc{L}}\frac{(g_V)_j(\tau)}{(c\tau+d)^r}\mbf{e}_j,
\end{equation}
 where $r$ is a large positive even integer and $\mc{L}$ is any set in $\Gamma$ containing all transformations with different lower rows. Now we note that if $V$ and $V^*$ have the same lower row, then we can write $V = Q^lV^*$ for some $l\in\ZZ$ and therefore we have
\begin{eqnarray*}
g_V(\tau) &=& (g_{Q^lV^*})(\tau)\\
&=& (g_{Q^l}|_{-k,\chi,\rho}V^*)(\tau) + g_{V^*}(\tau)\\
&=& (g_{V^*})(\tau),
\end{eqnarray*}
so that $\Phi_j(\tau;r)$ does not depend on the choice of coset representatives. Now based on the arguments in \cite{G,K} we prove the convergence of $\Phi(\tau;r)$ by using the following lemmas.

\begin{lem} \cite[Lemma 4]{K} \label{lemforpoin}
For real numbers $c,d$ and $\tau = u+iv\in\HH$, we have
\[\frac{v^2}{1+4|\tau|^2}(c^2+d^2) \leq |c\tau+d|^2 \leq 2(|\tau|^2 + v^{-2})(c^2+d^2).\]
\end{lem}

From now, we estimate $|g_V(\tau)|$ using the finite generators of $\Gamma$.
Suppose that
\[\{Q_0=Q,\cdots, Q_t, V_1,\cdots, V_s\}\]
is a fixed set of generators of $\Gamma$, including the $t+1$ parabolic generators $Q_0,\cdots, Q_t$, and the non-parabolic generators $V_1,\cdots, V_s$. If $\gamma\in\Gamma$, consider a factorization of $\gamma$ into {\it{sections}} (see \cite[pp. 156-157]{L}), $\gamma = C_1\cdots C_q$. Each section $C_i$ is either a non-parabolic generator of $\Gamma$ or a power of a parabolic generator of $\Gamma$. The importance of this factorization into sections lies in the result of Eichler \cite[Theorem 1]{Eic2} that, for any $\gamma = \sm a&b\\c&d\esm\in\Gamma$, the factorization can be carried out so that
\[q\leq m_1\log\mu(\gamma) + m_2,\]
where $m_1,m_2>0$ are independent of $\gamma$ and
\[\mu(\gamma) = a^2 + b^2 + c^2 + d^2.\]

We assume that the cocycle $\{g_\gamma|\ \gamma\in\Gamma\}$ in $\mc{P}$ satisfies
\begin{eqnarray} \label{generatorgrowth}
|(g_{V_i})_j(\tau)| &<& K(|\tau|^\rho + v^{-\sigma}),\ \text{for}\ 1\leq i\leq s,\ 1\leq j\leq p,\\
\nonumber |(g_i)_j(\tau)| &<& K(|\tau|^\rho + v^{-\sigma}),\ \text{for}\ 0\leq i\leq t,\ 1\leq j\leq p,
\end{eqnarray}
for positive constants $K,\rho$ and $\sigma$ which are independent of  particular generators involved.
Assume also that $2\sigma>-k$ and $\rho>k$.
Here, $g_i(\tau)$ is defined, by the definition of parabolic cocycle, as follows
\[g_{Q_i}(\tau) = (g_i|_{-k,\chi,\rho} Q_i)(\tau) - g_i(\tau).\]

\begin{lem} \label{lemforpoin2} If $\{g_\gamma|\ \gamma\in\Gamma\}$ is a parabolic cocycle, then there exists $K^*>0$ depending only upon $\Gamma$ and $\{g_\gamma|\ \gamma\in\Gamma\}$ such that
\[ |(g_{C_h}|_{-k,\chi,\rho} C_{h+1}\cdots C_q)_j(\tau)| \leq K^*\mu(\gamma)^e\lambda(\tau),\]
for $1\leq h\leq q,\ 1\leq j\leq p$. Here, $e = \max(\rho/2, \sigma+k/2)$ and $\gamma = C_1\cdots C_q$ is a factorization into sections of $\gamma\in\Gamma$ and $\lambda(\tau) = (|\tau|^2+v^{-2})^e\{\frac12 v^{2k-2\rho}+\frac12(1+4|\tau|^2)^{\rho-k}+v^{-\sigma}\}$.
\end{lem}

\begin{proof} [\bf Proof of Lemma \ref{lemforpoin2}]
Consider first the case when $C_h$ is a non-parabolic generator. Let $V = C_{h+1}\cdots C_q = \sm a&b\\c&d\esm$. By the definition of the slash operator, we have
\begin{eqnarray*}
|(g_{C_h}|_{-k,\chi,\rho} V)_j(\tau)| = |c\tau+d|^k|\bar{\chi}(V)|\sum_{l=1}^p |\rho(V^{-1})_{j,l}||(g_{C_h})_l(V\tau)|,
\end{eqnarray*}
where $\rho(V^{-1})_{jl}$ is the $(j,l)$th entry of $\rho(V^{-1})$. Then we see by (\ref{generatorgrowth}) that
\begin{eqnarray*}
|(g_{C_h}|_{-k,\chi,\rho} V)_j(\tau)| &\leq& |c\tau+d|^k\sum_{l=1}^p |(g_{C_h})_l(V\tau)|\\
&<& |c\tau+d|^k \cdot pK\{|V\tau|^\rho + v^{-\sigma}|c\tau+d|^{2\sigma}\}\\
&=& pK|a\tau+b|^\rho|c\tau+d|^{k-\rho} + pK|c\tau+d|^{2\sigma+k}v^{-\sigma}.
\end{eqnarray*}
By Lemma \ref{lemforpoin}, we obtain that
\begin{eqnarray*}
|a\tau+b|^\rho &\leq& 2^{\rho/2}(|\tau|^2 + v^{-2})^{\rho/2}(a^2+ b^2)^{\rho/2},\\
|c\tau+d|^{2\sigma+k} &\leq& 2^{\sigma+k/2}(|\tau|^2 + v^{-2})^{\sigma+k/2}(c^2+d^2)^{\sigma+k/2},
\end{eqnarray*}
and
\[|c\tau+d|^{k-\rho} \leq \biggl(\frac{v^2}{1+4|\tau|^2}\biggr)^{(k-\rho)/2}(c^2+d^2)^{(k-\rho)/2}.\]
Here, we used the assumption that $2\sigma>-k$ and $\rho>k$.
Hence we can check that
\begin{eqnarray*}
|(g_{C_h}|_{-k,\chi,\rho} V)_j(\tau)| &<& pK 2^{\rho/2}(|\tau|^2 + v^{-2})^{\rho/2}(a^2 + b^2)^{\rho/2}\biggl(\frac{1+4|\tau|^2}{v^2}\biggr)^{(\rho-k)/2}(c^2 + d^2)^{(k-\rho)/2}\\
&&+pK 2^{\sigma+k/2}(|\tau|^2 + v^{-2})^{\sigma+k/2}(c^2 + d^2)^{\sigma+k/2}v^{-\sigma}.
\end{eqnarray*}
Since the non-zero $c,\ \sm *&*\\c&*\esm\in\Gamma$, with $\Gamma$ discrete, have a positive lower bound, it follows that $c^2+d^2$ has a positive lower bound. Hence we get the following inequality
\begin{eqnarray*}
|(g_{C_h}|_{-k,\chi,\rho} V)_j(\tau)| &<& K_1(a^2+b^2)^{\rho/2}(|\tau|^2+v^{-2})^{\rho/2}\biggl(\frac{1+4|\tau|^2}{v^2}\biggr)^{(\rho-k)/2}\\
&&+K_1'(c^2+d^2)^{\sigma+k/2}(|\tau|^2+v^{-2})^{\sigma+k/2}v^{-\sigma}.
\end{eqnarray*}
By \cite[Theorem 2]{Eic2}, we have
\[a^2 + b^2 + c^2 + d^2 = \mu(V) \leq K_2\mu(\gamma),\]
so that
\begin{eqnarray*}
|(g_{C_h}|_{-k,\chi,\rho} V)_j(\tau)| &\leq& K_3\mu(\gamma)^{\rho/2}(|\tau|^2+v^{-2})^{\rho/2}v^{k-\rho}(1+4|\tau|^2)^{(\rho-k)/2}\\
&&+K_3'\mu(\gamma)^{\sigma+k/2}v^{-\sigma}(|\tau|^2+v^{-2})^{\sigma+k/2}.
\end{eqnarray*}
Letting $e = \max(\rho/2, \sigma+k/2)$, we see that
\begin{eqnarray*}
|(g_{C_h}|_{-k,\chi,\rho} V)_j(\tau)| &\leq& K_4\mu(\gamma)^e(|\tau|^2+v^{-2})^e\{v^{k-\rho}(1+4|\tau|^2)^{(\rho-k)/2}+v^{-\sigma}\}\\
&\leq& K_4\mu(\gamma)^e(|\tau|^2+v^{-2})^e\{\frac12 v^{2k-2\rho}+\frac12(1+4|\tau|^2)^{\rho-k}+v^{-\sigma}\}.
\end{eqnarray*}
We used the fact that $\{v^{k-\rho} - (1+4|\tau|^2)^{(\rho-k)/2}\}^2\geq0$.
If we put $\lambda(\tau) = (|\tau|^2+v^{-2})^e\{\frac12 v^{2k-2\rho}+\frac12(1+4|\tau|^2)^{\rho-k}+v^{-\sigma}\}$, then we have
\[ |(g_{C_h}|_{-k,\chi,\rho} V)_j(\tau)| \leq K_4\mu(\gamma)^e\lambda(\tau).\]


Now we deal with the case in which $C_h$ is a parabolic section, that is $C_h = Q_i^m$ for some $0\leq i\leq t$. Then
\[g_{Q_i}(\tau) = (g_i|_{-k,\chi,\rho}Q_i)(\tau) - g_i(\tau),\]
and so, by the consistency condition for the cocycle as in (\ref{cocyclevvm}), we also have
\[g_{C_h}(\tau) = (g_i|_{-k,\chi,\rho}C_h)(\tau) - g_i(\tau).\]
From this it follows that
\[(g_{C_h}|_{-k,\chi,\rho}C_{h+1}\cdots C_q)_j(\tau) = (g_i|_{-k,\chi,\rho}C_h\cdots C_q)_j(\tau) - (g_i|_{-k,\chi,\rho}C_{h+1}\cdots C_q)_j(\tau),\]
for $1\leq j\leq p$. The previous argument applies to each of the two terms on the right hand side to yield
\[|(g_{C_h}|_{-k,\chi,\rho} V)_j(\tau)| \leq K_5\mu(\gamma)^e\lambda(\tau).\]
The proof is complete.
\end{proof}

Now we need to look into $\mu(\gamma)$ to complete the estimation of $|g_V(\tau)|$ in (\ref{genpoin}).
It is helpful at this point to introduce a specific fundamental region. We employ the Ford fundamental region $\mc{R}$ defined as follows (see \cite[pp. 139]{Leh0})
\begin{equation} \label{ford}
\mc{R} := \{\tau\in\HH|\ |\Re\tau|<\lambda/2\ \text{and}\ |c\tau+d|>1\ \text{for all}\ \gamma = \sm *&*\\c&d\esm\in\Gamma-\Gamma_\infty\},
\end{equation}
where $\lambda$ is a constant defined as in (\ref{lambda}).
Then there exists $v_0>0$ with $iv_0\in\mc{R}$. Now determine $\mc{L}$ by the condition that $\gamma\in\mc{L}$ if $-\lambda/2 \leq \Re(\gamma(iv_0)) < \lambda/2$.

\begin{lem} \cite[Lemma 6]{K} \label{lemforpoin3}
If $\gamma = \sm a&b\\c&d\esm\in\mc{L}$, chosen as indicated above, then
\[\mu(\gamma) \leq K'(c^2+d^2),\]
for a positive constant $K'$, independent of $\gamma$.
\end{lem}

In the following theorem we prove by using Lemma \ref{lemforpoin}, \ref{lemforpoin2} and \ref{lemforpoin3} the convergence of $\Phi(\tau;r)$ for sufficiently large $r$.

\begin{thm} \label{convergepoin} Let $e = \max(\rho/2, \sigma+k/2)$, where $\rho$ and $\sigma$ are given in (\ref{generatorgrowth}).
If $r>2e+4$, then the generalized Poincar\'e series $\Phi(\tau;r)$, defined as in (\ref{genpoin}), converges absolutely and uniformly on compact subsets of $\HH$.
\end{thm}

\begin{proof} [\bf Proof of Theorem \ref{convergepoin}]
Suppose that $\gamma\in\mc{L}$. As before, we write $\gamma = C_1\cdots C_q$, a product of sections. Applying (\ref{cocyclevvm}) repeatedly, we have that
\begin{equation} \label{sumrepresentation}
g_{\gamma}(\tau) = g_{C_1\cdots C_q}(\tau) = (g_{C_1}|_{-k,\chi,\rho}C_2\cdots C_q)(\tau) + (g_{C_2}|_{-k,\chi,\rho}C_3\cdots C_q)(\tau) + \cdots + g_{C_q}(\tau)
\end{equation}
with $q\leq m_1\log\mu(\gamma) + m_2$ terms on the right hand side. We need to estimate the absolute value of general term of each component of the series (\ref{genpoin}). This is
\[\biggl|\frac{(g_\gamma)_j(\tau)}{(c\tau+d)^r}\biggr| = |(g_\gamma)_j(\tau)| |c\tau+d|^{-r}.\]
By (\ref{sumrepresentation}) and Lemma \ref{lemforpoin2} we have
\[|(g_\gamma)_j(\tau)|\leq K^*\mu(\gamma)^e\lambda(\tau)q \leq K_1^*\mu(\gamma)^{e+1}\lambda(\tau),\]
where $\lambda(\tau) = (|\tau|^2+v^{-2})^e\{\frac12 v^{2k-2\rho}+\frac12(1+4|\tau|^2)^{\rho-k}+v^{-\sigma}\}$, and we have used $q \leq m_1\log\mu(\gamma) + m_2\leq m_3\mu(\gamma)$ for a positive constant $m_3$, independent of $\gamma$. Lemma \ref{lemforpoin3} yields
\[|(g_\gamma)_j(\tau)|\leq  K_2^*(c^2+d^2)^{e+1}\lambda(\tau),\]
and by Lemma \ref{lemforpoin} we see that
\[|(g_\gamma)_j(\tau)|\leq K_2^*|c\tau+d|^{2e+2}\biggl(\frac{1+4|\tau|}{v^2}\biggr)^{e+1}\lambda(\tau).\]
Hence, we obtain that
\[|(g_\gamma)_j(\tau)||c\tau+d|^{-r} \leq K_2^*|c\tau+d|^{2e+2-r}\biggl(\frac{1+4|\tau|}{v^2}\biggr)^{e+1}\lambda(\tau).\]
With $r>2e+4$ it follows that $r-2e-2>2$ and hence by \cite[pp. 276-277]{Leh0} we see that a generalized Poincar\'e series $\Phi(\tau;r)$ converges absolutely and uniformly on compact subsets of $\HH$.
\end{proof}

Using a generalized Poincar\'e series $\Phi(\tau;r)$,  we can prove the existence of the vector-valued Eichler integral for given period functions. For this proof, we need the following result, which is a generalization of Petersson's result in \cite{Pet0}. For the notation, a vector-valued function $f(\tau)=\sum_{j=1}^p f_j(\tau)\mbf{e}_j$ is a meromorphic modular form of weight $-k$, multiplier system $\chi$ and type $\rho$ on $\Gamma$ if $f_j(\tau)$ is a meromorphic function on $\HH$ which satisfies two conditions in Definition \ref{dfnofvvm}.

\begin{lem} \label{genpet} Assume that $k$ is any real number. Let $\chi$ be a multiplier system of weight $-k$ and $\rho$ a unitary representation. Then there exists a vector-valued meromorphic modular form $f(\tau)$ of weight $-k$, multiplier system $\chi$ and type $\rho$ which has poles with given principal parts at finitely many points of $\bar{\mc{R}}\cap \HH$ and is  holomorphic elsewhere in $\bar{\mc{R}}$ with the possible exception of the cusps, where $\mc{R}$ is a Ford fundamental region as in (\ref{ford}).
\end{lem}

\begin{rmk}
In fact, we consider the principal part at a point $x$ in the Riemann surface $\Gamma\backslash \mathbb{H}$. Thus in this paper the principal part at $\tau_0 \in \mathbb{H}$ means the principal part at $\tau_0$ that can be expressed as
$$\sum_{n_0\leq n<0} c(n)(\tau-\tau_0)^{en},$$
where $e$ is the order of $\Gamma_{\tau_0}$.
\end{rmk}

\begin{proof} [\bf Proof of Lemma \ref{genpet}]
We want to construct a vector-valued function $I_{m, \tau_0, i}(\tau)$ such that
\begin{enumerate}
\item[(1)] a function $(I_{m,\tau_0,i})_m(\tau)$ has a pole of order $i$ at $\tau_0$ in $\mc{R}$ and is holomorphic everywhere else on $\mc{R}$,
\item[(2)] a function $(I_{m,\tau_0,i})_j(\tau)$ is holomorphic on $\HH$ if $j\neq m$.
\end{enumerate}

 First we construct a vector-valued modular form $H(\tau)$ in $M^!_{-k,\chi,\rho}(\Gamma)$. Knopp and Mason \cite{KM} defined a vector-valued Poincar\'e series $P(\tau) = P(\tau,\rho,r,\nu,m,j)$ in the following fashion. Fix $m$ an integer and $j,\ 1\leq j\leq p$, and let
\[P(\tau,\rho,r,\chi,m,j) = \frac12 \sum_{M\in\Gamma_\infty\setminus\Gamma}\frac{e^{2\pi i(m+\kappa_j)M\tau}}{\chi(M)(c\tau+d)^r}\rho^{-1}(M)\mbf{e}_j,\]
where $\chi$ is a multiplier system of weight $r$ and
\[\chi(Q)\rho(Q) = \sm e^{2\pi i\kappa_{1}}&&&&\\ &\cdot&&&\\ &&\cdot&&\\ &&&\cdot&\\ &&&&e^{2\pi i\kappa_{p}}\esm.\]
Here, $Q$ is a generator of $\Gamma_\infty$.
This Poincar\'e series is a vector-valued modular form in $M^!_{r,\chi,\rho}(\Gamma)$ (Knopp and Mason deals with a vector-valued Poincar\'e series $P(\tau,\rho,r,\nu,m,j)$ when $r$ is an even integer with $r>2$ and $\Gamma$ is a full modular group. But the same proof for the convergence applies to the case where $r$ is real and $\Gamma$ is a $H$-group).
Let $n$ be a positive integer such that $-k+12n>2$. Then by the above discussion, there is a vector-valued modular form $P(\tau)$ in $M^!_{-k+12n,\chi,\rho}(\Gamma)$. Then $H(\tau) := \frac{P(\tau)}{\Delta(\tau)^n}$ is a vector-valued modular form in $M^!_{-k,\chi,\rho}(\Gamma)$, where $\Delta(\tau) = e^{2\pi i\tau}\prod_{n=1}^\infty (1-e^{2\pi i n\tau})^{24}$ is a cusp form of weight $12$ with the trivial character on $\SL(2,\ZZ)$, which has no zeros or poles on $\HH$.

Next we want to modify the function $H(\tau)$ so that we get a desired function $I_{m,\tau_0,i}(\tau)$. Note that $H(\tau)=\sum_{j=1}^p H_j(\tau)\mbf{e}_j$ is holomorphic on $\HH$. Let $Z(H,m,\tau_0)$ be the order of zero of $H_m(\tau)$ at $\tau_0$. By the Petersson's result in \cite{Pet0}, we see that there is a (scalar-valued) meromorphic modular form $g(\tau)$ of weight $0$ with the trivial character on $\Gamma$ such that $g(\tau)$ has a pole of order $i+Z(H,m,\tau_0)$ at $\tau_0$ and is holomorphic elsewhere on $\mc{R}$. We can consider the function
\begin{equation*}
(I_{m,\tau_0,i})_j(\tau) :=
\begin{cases}
H_j(\tau) & \text{if $j\neq m$},\\
H_j(\tau) g(\tau) & \text{if $j=m$}.
\end{cases}
\end{equation*}
Then this is a desired function.

Now we construct a vector-valued meromorphic modular form $f(\tau)$ satisfying the desired property for the principal parts.
Suppose that the principal part at $\tau_0$ is given by
\begin{equation} \label{givenprin}
\sum_{n_0\leq n<0} c(n)(\tau-\tau_0)^{en},
\end{equation}
where $e$ is the order of $\Gamma_{\tau_0}$. It is enough to show that there is a vector-valued meromorphic modular form $f(\tau)=\sum_{j=1}^p f_j(\tau)\mbf{e}_j$ of weight $-k$, multiplier system $\chi$ and type $\rho$ on $\Gamma$ which has a pole  with a given principal part at $\tau_0$ in the $m$th component and is holomorphic elsewhere in $\mc{R}$.
Let $a(m,\tau_0,i) = \lim_{\tau\to\tau_0} (I_{m,\tau_0,i})_m(\tau)(\tau-\tau_0)^i$. If we let
\[ f(\tau) :=\sum_{n_0\leq n<0} \frac{c(n)}{a(m,\tau_0,n)}I_{m,\tau_0,n}(\tau),\]
then the principal part of $f_m(\tau)$ at $\tau_0$ is the same as given in (\ref{givenprin}) and $f_m(\tau)$ is holomorphic elsewhere in $\mc{R}$.
\end{proof}

 Following the literature on the Eichler cohomology theory, we introduce {\it{a left-finite expansion}} at each parabolic cusp for the consistency. The expansion at the cusp $i\infty$ has the form
\[F(\tau) = \sum_{j=1}^p \sum_{n\gg-\infty} a_0(n,j)e^{2\pi i(n+\kappa_{j,0})/\lambda_0}\mbf{e}_j,\]
where $Q=Q_0=\sm 1&\lambda_0\\0&1\esm,\ \lambda_0>0,$ is a generator of $\Gamma_\infty$ and
\begin{equation} \label{kappa}
\chi(Q)\rho(Q) = \sm e^{2\pi i\kappa_{1,0}}&&&&\\ &\cdot&&&\\ &&\cdot&&\\ &&&\cdot&\\ &&&&e^{2\pi i\kappa_{p,0}}\esm.
\end{equation}
Let $q_1,\cdots, q_t$ be the inequivalent parabolic cusps other than infinity. Suppose also that
\[Q_i = \sm *&*\\ c_i&d_i\esm,\ 1\leq i\leq t,\]
is a parabolic generator, i.e., $Q_i$ is a generator of $\Gamma_j$, where $\Gamma_j$ is the cyclic subgroup of $\Gamma$ fixing $q_i,\ 1\leq i\leq t$.
To describe the expansion at a finite parabolic cusp $q_i$, choose $A_i := \sm 0&-1\\ 1&-q_i\esm$, so that $A_i$ has determinant $1$ and $A_i(q_i) = \infty$. Then the {\it{width}} of the cusp $q_i$ is a positive real number such that
\[A_i^{-1}\sm 1&\lambda_i\\0&1\esm A_i = Q_i.\]
The expansion at the cusp $q_i$ has the form
\[F(\tau) = (\tau-q_i)^k\sum_{j=1}^p \sum_{n\gg-\infty} a_i(n,j)e^{\frac{-2\pi i(n+\kappa_{j,i})}{\lambda_i(\tau-q_i)}}\mbf{e}_j,\]
where
\begin{equation} \label{kappafinite}
\chi(Q_i)\rho(Q_i) = \sm e^{2\pi i\kappa_{1,i}}&&&&\\ &\cdot&&&\\ &&\cdot&&\\ &&&\cdot&\\ &&&&e^{2\pi i\kappa_{p,i}}\esm,
\end{equation}
for $0\leq \kappa_{j,i}<1,\ 1\leq j\leq p$ and $1\leq i\leq t$. (The motivation for the form of these expansions can be found in \cite[pp. 17-20]{Kno2}.)

 With these notations we can prove the existence of the vector-valued Eichler integral for given period functions.

\begin{thm} \label{converse}
Assume that $k$ is any real number. Let $\chi$ be a multiplier system of weight $-k$ and $\rho$ a unitary representation. Suppose $\{g_\gamma|\ \gamma\in\Gamma\}$ is a parabolic cocycle of weight $-k$, multiplier system $\chi$ and type $\rho$ on $\Gamma$ in $\mc{P}$.
Then there exists a vector-valued function $\Phi(\tau)$, holomorphic in $\HH$, such that
\[(\Phi|_{-k,\chi,\rho}\gamma)(\tau) - \Phi(\tau) = g_\gamma(\tau)\]
for all $\gamma\in\Gamma$.
\end{thm}

\begin{rmk} \label{leftexpansion}
Since $\{g_\gamma|\ \gamma\in\Gamma\}$ is a parabolic cocycle, there is an element $g_i(\tau)\in\mc{P}$ such that
\[g_{Q_i}(\tau) = (g_i|_{-k,\nu,\rho}Q_i)(\tau) - g_i(\tau),\]
for $0\leq i\leq t$.
Then $\Phi(\tau) - g_i(\tau)$ is invarinat under the slash operator $|_{-k,\chi,\rho}Q_i$ and hence has the expansions at parabolic cusps $q_i,\ 0\leq i\leq t$, of the forms
\begin{eqnarray*}
\Phi(\tau) &=& g_i(\tau) + (\tau-q_i)^k\sum_{j=1}^p \sum_{n\gg-\infty}a_i(n,j)e^{\frac{-2\pi i(n+\kappa_{j,i})}{\lambda_i(\tau-q_i)}}\mbf{e}_j,\ 1\leq i\leq t,\\
\Phi(\tau) &=& g_0(\tau) + \sum_{j=1}^p \sum_{n\gg-\infty} a_0(n,j)e^{\frac{2\pi i(n+\kappa_{j,0})}{\lambda_0}}\mbf{e}_j,\ i=0.
\end{eqnarray*}
\end{rmk}

\begin{proof} [\bf Proof Theorem \ref{converse}]
We prove this theorem by using a vector-valued generalized Poincar\'e series $\Phi(\tau;r)$. To define $\Phi(\tau;r)$, we need the condition that $g_Q(\tau)=0$. But that is not true in general. So we modify the given cocycle $\{g_\gamma|\ \gamma\in\Gamma\}$ as follows.
For $\gamma\in\Gamma$, put
\[g_\gamma^*(\tau) = g_\gamma(\tau) - (g_0|_{-k,\chi,\rho}\gamma-g_0)(\tau).\]
Then $\{g_\gamma^*|\ \gamma\in\Gamma\}$ is again a parabolic cocycle in $\mc{P}$ and now $g_Q^*(\tau) = 0$. Thus we may form the vector-valued generalized Poincar\'e series $\Phi^*(\tau;r)$ as
\[\Phi^*(\tau;r) := \sum_{j=1}^p \sum_{V=\sm a&b\\c&d\esm\in\mc{L}} \frac{(g_V^*)_j(\tau)}{(c\tau+d)^r}\mbf{e}_j\]
for sufficiently large $r$ so that $\Phi^*(\tau;r)$ converges.

First we investigate the transformation properties of the vector-valued generalized Poincar\'e series $\Phi^*(\tau;r)$.
Note that for every $M\in\Gamma$ there is a ono-to-one correspondence between $\mc{L}$ and $\mc{L}M$. Therefore, using the absolute convergence and the fact that $r\in 2\ZZ_{>0}$,
 for $M = \sm \alpha&\beta\\ \gamma&\delta\esm, VM = \sm \tilde{\alpha}&\tilde{\beta}\\ \tilde{\gamma}&\tilde{\delta}\esm\in\Gamma$, we have
\begin{eqnarray*}
(\Phi^*|_{-k,\chi,\rho}M)(\tau;r)&=&\chi^{-1}(M)(\gamma\tau+\delta)^k\rho^{-1}(M)\Phi^*(M\tau;r)\\
&=&\sum_{V=\sm a&b\\c&d\esm\in\mc{L}}\chi^{-1}(M)(\gamma\tau+\delta)^k\rho^{-1}(M)g_V^*(M\tau)(cM\tau+d)^{-r}\\
&=&\sum_{V=\sm a&b\\c&d\esm\in\mc{L}}(g_V^*|_{-k,\chi,\rho}M)(\tau)(cM\tau+d)^{-r}\\
&=&\sum_{V=\sm a&b\\c&d\esm\in\mc{L}}(g_{VM}^*(\tau) - g_M^*(\tau))(cM\tau+d)^{-r}\\
&=&(\gamma\tau+\delta)^r\sum_{V=\sm a&b\\c&d\esm\in\mc{L}}(g_{VM}^*(\tau)-g_M^*(\tau))(\tilde{\gamma}\tau+\tilde{\delta})^{-r}\\
&=&(\gamma\tau+\delta)^r(\Phi^*(\tau;r) - \psi(\tau;r)g_M^*(\tau)).
\end{eqnarray*}
Here $\psi(\tau;r)$ is the classical Eisenstein series
\[\psi(\tau;r) = \sum_{V=\sm a&b\\c&d\esm\in\mc{L}}(c\tau+d)^{-r},\]
which converges absolutely for $r>2$. Also we have the transformation law
\[\psi(M\tau;r) = (\gamma\tau+\delta)^r\psi(\tau;r)\]
for all $M=\sm \alpha&\beta\\\gamma&\delta\esm \in\Gamma$.

Next we define the vector-valued function $F^*(\tau)$ using $\Phi^*(\tau;r)$ and $\psi(\tau;r)$ in the following way
\[F^*(\tau) := -\frac{\Phi^*(\tau;r)}{\psi(\tau;r)}.\]
Then we get
\begin{eqnarray*}
(F^*|_{-k,\chi,\rho}M)(\tau) &=& -\frac{(\Phi^*|_{-k,\chi,\rho}M)(\tau)}{\psi(M\tau;r)}= -\frac{\Phi^*(\tau;r)}{\psi(\tau;r)} + g_M^*(\tau)= F^*(\tau) + g_M^*(\tau).
\end{eqnarray*}
Defining $F(\tau) := F^*(\tau) +g_0(\tau)$, we have
\[(F|_{-k,\chi,\rho}M)(\tau) -F(\tau) = g_M^*(\tau) + (g_0|_{-k,\chi,\rho}M)(\tau) - g_0(\tau) = g_M(\tau)\]
for $M\in\Gamma$. Therefore, $F(\tau)$ is the vector-valued Eichler integral satisfying the desired transformation properties.

However, $F(\tau)$ may have a pole at $\mc{R}$ because of $\psi(\tau;r)$. So we modify
 the function $F(\tau)$ by substracting a vector-valued meromorphic modular form $f(\tau)$ with the same principal part at all poles of $F(\tau)$ on $\mc{R}$  using Lemma \ref{genpet},  so that we get an vector-valued Eichler integral $F(\tau)-f(\tau)$ with poles only at cusps and the same periods as $F(\tau)$.
 Therefore, $F(\tau)-f(\tau)$ is holomorphic on $\HH$ and for all $M\in\Gamma$ we have that
\[((F-f)|_{-k,\chi,\rho}M)(\tau) = (F-f)(\tau) + g_M(\tau).\]
\end{proof}

\medskip

\subsection{Parabolicity of cocycles in $H^1_{-k,\chi,\rho}(\Gamma,\mc{P})$} \label{section3.4}
In this subsection we show that every cocycle in $\mc{P}$ is parabolic. To prove this we need the following lemma.

\begin{lem} \label{taylor}
Suppose that $g(\tau)=\sum_{j=1}^pg_i(\tau)\mbf{e}_i\in\mc{P}$. For a unitary $p$ by $p$ matrix $C$, there exists $f(\tau)=\sum_{j=1}^p f_j(\tau)\mbf{e}_j\in\mc{P}$ such that
\begin{equation} \label{vectoreqn}
 C f(\tau+1) - f(\tau) = g(\tau),\ \tau\in\HH.
\end{equation}
\end{lem}

\begin{proof} [\bf Proof of Lemma \ref{taylor}]
Since $C$ is unitary, it is diagonalizable. This implies that there are $p$ by $p$ matrices $P$ and $D$ such that
\[C = PDP^{-1},\]
where $P$ is unitary and $D$ is unitary and diagonal. Furthermore, every diagonal entry of $D$ is of absolute value $1$. Recall that $\{\mbf{e}_1,\cdots, \mbf{e}_p\}$ is the standard basis of $\CC^p$. Let $\mbf{e}'_j = P\mbf{e}_j$ for $1\leq j\leq p$. Then $\{\mbf{e}'_1,\cdots, \mbf{e}'_p\}$ is also a basis of $\CC^p$ and
there are scalar-valued functions $g'_1(\tau), \cdots, g'_p(\tau)$ such that $g(\tau) = \sum_{j=1}^p g'_j(\tau)\mbf{e}'_j$.
Then since $g'_j(\tau)$ is a linear combination of $g_1(\tau),\cdots, g_p(\tau)$, it satisfies the growth condition (\ref{growthconditionforp}).
Then solving the equation (\ref{vectoreqn}) is equivalent to finding scalar-valued functions $f'_1(\tau),\cdots, f'_p(\tau)$ such that $\sum_{j=1}^p f'_j(\tau)\mbf{e}'_j\in \mc{P}$ and
\begin{equation} \label{modifiedeqn}
 D \biggl(\sum_{j=1}^p f'_j(\tau+1)\mbf{e}'_j\biggr) - \sum_{j=1}^p f'_j(\tau)\mbf{e}'_j = \sum_{j=1}^p g'_j(\tau)\mbf{e}'_j.
\end{equation}
If $d_j$ is a $j$th diagonal entry of $D$, then by Proposition 9 in \cite{K} we can find $f'_j(\tau)$ satisfying the growth condition (\ref{growthconditionforp}) such that
\[ d_j f'_j(\tau+1) - f'_j(\tau) = g'_j(\tau),\]
for $1\leq j\leq p$ and $\tau\in\HH$. Then $\sum_{j=1}^p f'_j(\tau)\mbf{e}'_j(\tau)$ is an element of $\mc{P}$, which is a desired solution of (\ref{modifiedeqn}). Finally, we can find scalar-valued functions $f_1(\tau),\cdots f_p(\tau)$ satisfying (\ref{growthconditionforp}) such that
\[\sum_{j=1}^p f_j(\tau)\mbf{e}_j = \sum_{j=1}^p f'_j(\tau)\mbf{e}'_j.\]
This completes the proof.
\end{proof}

With this lemma, we can prove that every cocycle in $\mc{P}$ is parabolic.

\begin{thm} \label{parabolicity}
Let $k\in \RR$ and $\chi$ a multiplier system of weight $-k$. Suppose that $\rho$ is a $p$-dimensional unitary representation. Then
\[H^1_{-k,\chi,\rho}(\Gamma,\mc{P}) = \tilde{H}^1_{-k,\chi,\rho}(\Gamma,\mc{P}).\]
\end{thm}

\begin{proof} [\bf Proof of Theorem \ref{parabolicity}]
By definition, it is enough to show that for a parabolic element $Q\in\Gamma$ and a vector-valued function $g(\tau)\in \mc{P}$, there exists $f(\tau)\in \mc{P}$ such that
\begin{equation} \label{claimeqn}
(f|_{-k,\chi,\rho}Q)(\tau) - f(\tau) = g(\tau).
\end{equation}

First suppose  that $Q$ is a translation, that is, $Q = \sm 1&\lambda\\0&1\esm,\ \lambda>0$.
Put $\varphi(\tau) := g(\lambda\tau)$. Then $\varphi(\tau)\in\mc{P}$ and we may apply Lemma \ref{taylor} to conclude that there exists $\psi(\tau)\in\mc{P}$ such that
\[\bar{\chi}(Q)\rho(Q)^{-1}\psi(\tau+1) - \psi(\tau) = \varphi(\tau).\]
If we put $f(\tau) := \psi(\tau/\lambda)$, then $f(\tau)\in\mc{P}$ and
\begin{eqnarray*}
(f|_{-k,\chi,\rho}Q)(\tau) - f(\tau) &=& \bar{\chi}(Q)\rho(Q)^{-1}f(\tau+\lambda) - f(\tau)\\
&=&\bar{\chi}(Q)\rho(Q)^{-1} \psi(\tau/\lambda+1) - \psi(\tau/\lambda) = \varphi(\tau/\lambda) = g(\tau).
\end{eqnarray*}
Thus (\ref{claimeqn}) has a solution if $Q$ is a translation.

Next, if $Q\in\Gamma$ is parabolic, but not a translation, then $Q$ can be written in the form $Q = A^{-1}\sm 1&\lambda\\0&1\esm A$, $\lambda>0$, where $A \in \SL(2,\RR)$ such that $A q = \infty$. Here, $q$ is a parabolic point fixed by $Q$. Suppose that $Q = \sm *&*\\ c&d\esm,\ A = \sm *&*\\ \gamma&\delta\esm$ and put
\[\varphi(\tau) := (\gamma A^{-1}\tau+\delta)^{-k} g(A^{-1}\tau),\]
and
\[\epsilon = \bar{\chi}(Q)(c\tau+d)^k(\gamma Q\tau+\delta)^k(\gamma\tau+\delta)^{-k}.\]
Since $(c\tau+d)(\gamma Q\tau+\delta) = \gamma\tau+\delta$, we see that $|\epsilon| = 1$.
One can check that $\varphi(\tau)\in\mc{P}$ and hence by the previous case there exists $\psi(\tau)\in\mc{P}$ such that
\[\epsilon\rho(Q)^{-1}\psi(\tau+\lambda) - \psi(\tau) = \varphi(\tau).\]
If we put $f(\tau) := (\gamma\tau+\delta)^k\psi(A\tau)$, then $f(\tau)\in\mc{P}$ and we see that
\begin{eqnarray*}
(f|_{-k,\chi,\rho}Q)(\tau) - f(\tau) &=& \bar{\chi}(Q)\rho(Q)^{-1}(c\tau+d)^kf(Q\tau) - f(\tau)\\
&=& \bar{\chi}(Q)\rho(Q)^{-1}(c\tau+d)^k(\gamma Q\tau+\delta)^k\psi(AQ\tau) - f(\tau)\\
&=& \epsilon\rho(Q)^{-1}(\gamma\tau+\delta)^k \psi(A\tau+\lambda) - (\gamma\tau+\delta)^k\psi(A\tau)\\
&=& (\gamma\tau+\delta)^k\varphi(A\tau)\\
&=& g(\tau).
\end{eqnarray*}
Here, we used that $AQ = \sm 1&\lambda\\0&1\esm A$ so that $AQ\tau = A\tau + \lambda$.
Thus (\ref{claimeqn}) has a solution $f(\tau)\in\mc{P}$ for any parabolic $Q\in\Gamma$. This completes the proof.
\end{proof}

\subsection{The Eichler cohomology theorem for vector-valued modular forms} \label{section3.5}
In this subsection we state  the Eichler cohomology theorem for vector-valued modular forms.
The Eichler cohomology theorem states that there is an isomorphism between the cohomology group of $\Gamma$ and the space of vector-valued cusp forms. To define a map from vector-valued cusp forms to cocycles in $\mc{P}$ first we prove the following lemma.

\begin{lem} \label{lemforinj} Let $g(\tau)\in S_{k+2,\bar{\chi},\bar{\rho}}(\Gamma)$ and put
\begin{equation} \label{automorphicintegral}
G(\tau) := \biggl[\int_{i\infty}^\tau g(w)(w-\bar{\tau})^kdw\biggr]^-,
\end{equation}
where $\tau\in\HH$ and $[\ ]^-$ indicates the complex conjugate of the function inside $[\ ]$. Then
\[(G|_{-k,\chi,\rho}\gamma)(\tau) - G(\tau) = \biggl[\int^{i\infty}_{\gamma^{-1}(i\infty)}g(w)(w-\bar{\tau})^kdw\biggr]^-\]
for all $\gamma\in\Gamma$. Moreover, if for each $\gamma\in\Gamma$ we set
\[g_\gamma(\tau) :=  \biggl[\int^{i\infty}_{\gamma^{-1}(i\infty)}g(w)(w-\bar{\tau})^kdw\biggr]^-,\]
then the collection $\{g_\gamma|\ \gamma\in\Gamma\}$ is a cocycle in $\mc{P}$.
\end{lem}

\begin{proof}[\bf Proof of Lemma \ref{lemforinj}]
This is essentially Lemma 2.2 of \cite{KM2}.
First we have
\begin{eqnarray*}
(G|_{-k,\chi,\rho}\gamma)(\tau) - G(\tau) &=& \bar{\chi}(\gamma)\rho^{-1}(\gamma)(c\tau+d)^kG(\gamma\tau) - G(\tau)\\
&=& \biggl[\chi(\gamma)\bar{\rho}^{-1}(\gamma)(c\bar{\tau}+d)^k\int_{i\infty}^{\gamma\tau}g(w)(w-\gamma\bar{\tau})^kdw - \int_{i\infty}^\tau g(w)(w-\bar{\tau})^kdw\biggr]^-
\end{eqnarray*}
for $\gamma = \sm a&b\\c&d\esm\in\Gamma$. Then using the substitution $w\to \gamma w$ in the first integral and the fact that $g(\tau)\in S_{k+2,\bar{\chi},\bar{\rho}}(\Gamma)$, we obtain that
\begin{eqnarray*}
&& (G|_{-k,\chi,\rho}\gamma)(\tau) - G(\tau) \\
&&= \biggl[\chi(\gamma)\bar{\rho}^{-1}(\gamma)(c\bar{\tau}+d)^k\int_{\gamma^{-1}(i\infty)}^\tau g(\gamma w)(\gamma w-\gamma\bar{\tau})^k (cw+d)^{-2}dw -\int_{i\infty}^\tau g(w)(w-\bar{\tau})^kdw\biggr]^-\\
&&= \biggl[\int_{\gamma^{-1}(i\infty)}^\tau (c\bar{\tau}+d)^k(cw+d)^kg(w)(\gamma w-\gamma\bar{\tau})^kdw - \int_{i\infty}^\tau g(w)(w-\bar{\tau})^kdw\biggr]^-\\
&&= \biggl[\int_{\gamma^{-1}(i\infty)}^\tau (c\bar{\tau}+d)^k(cw+d)^kg(w)\biggl(\frac{w-\bar{\tau}}{(cw+d)(c\bar{\tau}+d)}\biggr)^kdw - \int_{i\infty}^\tau g(w)(w-\bar{\tau})^k dw\biggr]^-.
\end{eqnarray*}
Note that
\[\biggl(\frac{w-\bar{\tau}}{(cw+d)(c\bar{\tau}+d)}\biggr)^k = \frac{(w-\bar{\tau})^k}{(cw+d)^k(c\bar{\tau}+d)^k}.\]
Hence we conclude that
\begin{eqnarray*}
 (G|_{-k,\chi,\rho}\gamma)(\tau) - G(\tau) &=& \biggl[\int_{\gamma^{-1}(i\infty)}^\tau g(w)(w-\bar{\tau})^kdw - \int_{i\infty}^\tau g(w)(w-\bar{\tau})^kdw\biggr]^-\\
&=& \biggl[\int_{\gamma^{-1}(i\infty)}^{i\infty} g(w)(w-\bar{\tau})^kdw\biggr]^-.
\end{eqnarray*}
Since $g_\gamma(\tau) = (G|_{-k,\chi,\rho}\gamma)(\tau) - G(\tau)$, one can check that $\{g_\gamma|\ \gamma\in\Gamma\}$ is a cocycle in $\mc{P}$.
\end{proof}

Now we define a mapping
\[\eta: S_{k+2,\bar{\chi},\bar{\rho}}(\Gamma) \to H^1_{-k,\chi,\rho}(\Gamma,\mc{P})\]
by $\eta(g) = <g_\gamma|\ \gamma\in\Gamma>$, where $<g_\gamma|\ \gamma\in\Gamma>$ is the class determined by the cocycle
$\{g_\gamma|\ \gamma\in\Gamma\}$ constructed in the previous lemma. One can see that $\eta$ is a linear mapping. In this subsection, we just state the Eichler cohomology theorem for vector-valued modular forms.

\begin{thm} \label{mainvvmreal} Let $\rho$ be a $p$-dimensional unitary representation. Then for any real number $k$ and multiplier system $\chi$ of weight $-k$, we have an isomorphism
\[ \eta: S_{k+2,\bar{\chi},\bar{\rho}}(\Gamma) \cong H^1_{-k,\chi,\rho}(\Gamma,\mc{P}).\]
\end{thm}

Combining Theorem \ref{parabolicity} and Theorem \ref{mainvvmreal}, we have the following result.

\begin{thm} \label{mainvvmreal2}
Let $\rho$ be a $p$-dimensional unitary representation. Then for any real number $k$ and multiplier system $\chi$ of weight $-k$, we have an isomorphism
\[  S_{k+2,\bar{\chi},\bar{\rho}}(\Gamma) \cong \tilde{H}^1_{-k,\chi,\rho}(\Gamma,\mc{P}).\]
\end{thm}

In the following subsections, we will prove Theorem \ref{mainvvmreal} based on the arguments of Knopp and Mawi as in \cite{KM2}.

\medskip

\subsection{Injectivity of $\eta$} \label{section3.6} In this subsection, we prove the injectivity of $\eta$.
To prove that $\eta$ is injective, assume that $g(\tau) \in \Ker\eta$.
It suffices to show that $g\equiv 0$.
Since $g(\tau)\in \Ker\eta$, we see that $\eta(g) = <g_\gamma|\ \gamma\in\Gamma>$ is a coboundary. Thus there exists $p(\tau)\in\mc{P}$ such that
\[g_\gamma(\tau) = (p|_{-k,\chi\rho}\gamma)(\tau) - p(\tau)\]
for all $\gamma\in\Gamma$. With $G(\tau)$ defined as in (\ref{automorphicintegral}), we obtain that
\[((G-p)|_{-k,\chi,\rho}\gamma)(\tau) = (G - p)(\tau)\]
for all $\gamma\in\Gamma$.
Since $p(\tau)$ and $g(\tau)$ are holomorphic, we have
\[\frac{\partial}{\partial\bar{\tau}}(g_j(\tau)G_j(\tau)) = |g_j(\tau)|^2(\bar{\tau}-\tau)^k\]
and
\[\frac{\partial}{\partial\bar{\tau}}(p_j(\tau)g_j(\tau)) = 0\]
for each $j\in\{1,2,\cdots, p\}$. Let $\mc{F}$ be a fundamental region for $\Gamma$ in $\HH$. Then
\begin{equation} \label{stoke}
\int_{\mc{F}}|g_j(\tau)|^2(\bar{\tau}-\tau)^kdudv = \int_{\mc{F}} \frac{\partial}{\partial\bar{\tau}}(g_j(\tau)(G_j(\tau)-p_j(\tau)))dudv,
\end{equation}
where $\tau = u + iv\in\HH$. If we apply the Stoke theorem to the right hand side of (\ref{stoke}) and take a sum, then we have
\begin{equation} \label{stoke2}
\int_{\mc{F}}\sum_{j=1}^p|g_j(\tau)|^2(\bar{\tau}-\tau)^kdudv = \frac{-i}{2} \int_{\partial\mc{F}} \sum_{j=1}^p g_j(\tau)(G_j(\tau)-p_j(\tau))d\tau.
\end{equation}
Note that $\sum_{j=1}^p g_j(\tau)(G_j-p_j)(\tau)$ can be understood as an inner product on $p$-dimensional vectors. Since $\rho$ is a unitary representation, $\rho(\gamma)$ is a unitary matrix for any $\gamma\in\Gamma$ and hence it follows that
\[\sum_{j=1}^p g_j(\tau)(G_j(\tau)-p_j(\tau))d\tau\]
is $\Gamma$-invariant. Using this transformation property we can show that the right hand side of (\ref{stoke2}) is zero, from which we conclude that $g_j\equiv 0$ for each $j\in\{1,\cdots, p\}$. This implies that $\eta$ is injective.

\medskip

\subsection{Petersson's principal parts condition and related results} \label{section3.7}
In this subsection, we state and prove some results which we need to prove that $\eta$ is surjective.
Let $k\in\RR$. Let $\{\varphi_1,\cdots, \varphi_d\}$ be a basis of $S_{k+2,\bar{\chi},\bar{\rho}}(\Gamma)$
whose expansions at the finite cusps $q_i$ are given by
\[\varphi_l(\tau) = (\tau-q_i)^{-k-2}\sum_{j=1}^p \sum_{n\gg-\infty} a_{l,i}(n,j)e^{\frac{-2\pi i(n+\kappa'_{j,i})}{\lambda_{i}(\tau-q_i)}}\mbf{e}_j,\ 1\leq i\leq t,\]
and whose expansion at infinity is given by
\[\varphi_l(\tau) = \sum_{j=1}^p \sum_{n\gg-\infty}a_{l,0}(n,j)e^{2\pi i(n+\kappa'_{j,0})\tau/\lambda_0}\mbf{e}_j,\]
where $\kappa_{j,i},\ 1\leq j\leq p,\ 0\leq i\leq t,$ are non-zero constants given as in (\ref{kappa}) , (\ref{kappafinite}) and $\kappa'_{j,i},\ 1\leq j\leq p,\ 0\leq i\leq t,$ are given by
\begin{equation*}
\kappa'_{j,i} =
\begin{cases}
0 & \text{if $\kappa_{j,i} = 0$},\\
1-\kappa_{j,i} & \text{if $\kappa_{j,i}>0$}.
\end{cases}
\end{equation*}
The following result is a generalization of Petersson's result in \cite[pp. 388-389]{P}.

\begin{thm} \label{petersson} There exists $g(\tau)\in M^!_{-k,\chi,\rho}(\Gamma)$
whose expansion at the finite cusps $q_i$ are given by
\[g(\tau) = (\tau-q_i)^k\sum_{j=1}^p \sum_{n\gg-\infty}b_i(n,j)e^{\frac{-2\pi i(n+\kappa_{j,i})}{\lambda_i(\tau-q_i)}}\mbf{e}_j,\ 1\leq i\leq t,\]
and whose expansion at infinity is given by
\[g(\tau) = \sum_{j=1}^p \sum_{n\gg-\infty}b_0(n,j)e^{2\pi i(n+\kappa_{j,0})\tau/\lambda_0}\mbf{e}_j\]
if and only if
\[\Res(g\varphi_l) := \sum_{i=0}^t \Res((g\varphi_l)_i,q_i)=0,\]
for each $l = 1,\cdots, d$, where $\Res((g\varphi_l)_i,q_i) := \sum_{j=1}^p \sum_{n+\kappa_j<0} b_i(n,j)a_{l,i}(-(\kappa_{j,i}+\kappa'_{j,i})-n,j)$.
\end{thm}

\begin{proof} [\bf Proof of Theorem \ref{petersson}] This follows from the same argument in Theorem 3.1 in \cite{B2}.
\end{proof}

Since $\dim S_{k+2,\bar{\chi},\bar{\rho}}(\Gamma) = d$,
there exist pairs of natural numbers
\begin{equation} \label{related1}
(s_1,t_1),\cdots, (s_d,t_d)
\end{equation}
 such that a $d$ by $d$ matrix $C = [c_{lj}]$, where $c_{lj} = a_{l,0}(s_j,t_j)$, is nonsingular. Using this, we can prove the following result, which plays an important role in the proof of the surjectivity of $\eta$.

\begin{prop} \label{related2}
Let $k\in\RR$.
Given poles of prescribed principal parts $P_i$ at the cusp $q_i$ for $i = 0,\cdots, t$, there exists a vector-valued weakly holomorphic modular form  $g(\tau)$ in $M^!_{-k,\chi,\rho}(\Gamma)$ whose principal part at the finite cusp $q_i,\ 1\leq i\leq t$, is $P_i$ and whose principal part at the infinite cusp $q_0$ differs from $P_0$ by at most $d$ terms, where $d = \dim S_{k+2,\bar{\chi},\bar{\rho}}(\Gamma)$.
\end{prop}

\begin{proof} [\bf Proof of Proposition \ref{related2}]
Let  $(s_l,t_l),\ 1\leq l\leq d$, be as in (\ref{related1}) and $C = [c_{lj}]$, where $c_{lj} = a_{l,0}(s_j,t_j)$ and $a_{l,0}(s_j,t_j)$ are Fourier coefficients of $\varphi_l(\tau)$.
We shall construct $g(\tau)\in M^!_{-k,\chi,\rho}(\Gamma)$ whose principal part at $q_i$ for $1\leq i\leq t$ is $P_i$ and whose principal part at $q_0$ is given by
\[g(\tau) = P_0 + \sum_{l=1}^d a(s_l,t_l)e^{2\pi i(-s_l-\kappa'_{t_l,0})/\lambda_0}\mbf{e}_{t_l}.\]

By using Theorem \ref{petersson} we see that the necessary and sufficient condition for $g(\tau)$ to exist is that there is a solution $A= (a(s_1,t_1),\cdots, a(s_d,t_d))$ of the equation
\begin{equation} \label{3.5}
CA = B,
\end{equation}
where $B = (b_1,\cdots, b_d)$ with $b_l = -\sum_{i=0}^t\Res(P_i(\varphi_l)_i,q_i)$.
But since $C$ is nonsingular, (\ref{3.5}) has a unique solution for $A$. With these values for $a(s_1,t_1),\cdots, a(s_d,t_d)$, the condition in Theorem \ref{petersson} is satisfied by $g(\tau)$, and hence $g(\tau)$ exists.
\end{proof}

\medskip

\subsection{Surjectivity of $\eta$} \label{section3.8}
In this subsection, we prove the surjectivity of $\eta$. Assume that the dimension of the vector space $S_{k+2,\bar{\chi},\bar{\rho}}(\Gamma)$ is $d$. Since $\eta$ is one-to-one and $H^1_{-k,\chi,\rho}(\Gamma,\mc{P}) \cong \tilde{H}^1_{-k,\chi,\rho}(\Gamma,\mc{P})$, it suffices to show that $\dim \tilde{H}^1_{-k,\chi,\rho}(\Gamma,\mc{P})$ is at most $d$.

If $d = 0$, then by Theorem \ref{petersson}  there exists a $g(\tau)\in M^!_{-k,\chi,\rho}(\Gamma)$ with any preassigned principal parts at each of the cusps. Given a parabolic cocycle $\{g_\gamma| \gamma\in\Gamma\}$, we obtain $\Phi(\tau)$ whose expansions at the cusps have principal parts as given in Remark \ref{leftexpansion}. Then we use Theorem \ref{petersson} to obtain  $g(\tau)\in M^!_{-k,\chi,\rho}(\Gamma)$
whose expansions at the cusps have principal parts which agree with those of $\Phi(\tau)$.
 Now set $\Phi^*(\tau) = \Phi(\tau) - g(\tau)$. Then $(\Phi^*|_{-k,\chi,\rho}\gamma)(\tau)-\Phi^*(\tau) = g_\gamma(\tau)$ for all $\gamma\in\Gamma$ and $\Phi^*(\tau)\in\mc{P}$. Thus every parabolic cocycle is a coboundary, and hence the dimension of $\tilde{H}^1_{-k,\chi,\rho}(\Gamma,\mc{P})$ is zero. This completes the proof of Theorem \ref{mainvvmreal} when $d=0$.

 Now suppose that $\dim S_{k+2,\bar{\chi},\bar{\rho}}(\Gamma) = d>0$. Since $\eta$ is one-to-one, we know that there exist at least $d$ linearly independent elements in $\tilde{H}^1_{-k,\chi,\rho}(\gamma,\mc{P})$. Let $d$ of these linearly independent elements be
 \[<g^1_\gamma|\ \gamma\in\Gamma>, \cdots, <g^d_\gamma|\ \gamma\in\Gamma>,\]
 and let
 \[\{g^1_\gamma|\ \gamma\in\Gamma\}, \cdots, \{g^d_\gamma|\ \gamma\in\Gamma\}\]
 be responsible representatives. Corresponding to each of these parabolic cocycles, by using Theorem \ref{converse}, we obtain $\Phi_l(\tau),\ 1\leq l\leq d$.

 By Proposition \ref{related2}, for each $\Phi_l(\tau)$, there exists $f_l(\tau)\in M^!_{-k,\chi,\rho}(\Gamma)$ such that except possible $d$ terms in the expansion of $f_l(\tau)$ at the infinite cusp $f_l(\tau)$ has the same principal parts as $\Phi_l(\tau)$ in all the expansions at the cusps. Thus, if we let $\Phi^*_l(\tau) = \Phi_l(\tau)-f_l(\tau)$, then after renaming coefficients we have
 \begin{eqnarray*}
 \Phi^*_l(\tau) &=& g_i(\tau) + (\tau-q_i)^k\sum_{j=1}^p \sum_{n+\kappa_{j,i}\geq0}a_i^l(n,j)e^{\frac{-2\pi i(n+\kappa_{j,i})}{\lambda_i(\tau-q_i)}}\mbf{e}_j,\ 1\leq i\leq t,\\
\Phi^*_l(\tau) &=& g_0(\tau) + \sum_{l=1}^d a^l(s_l,t_l)e^{2\pi i(-s_l-\kappa'_{t_l,0})/\lambda_0}\mbf{e}_{t_l} + \sum_{j=1}^p \sum_{n+\kappa_{j,0}\geq0} a_0^l(n,j)e^{2\pi i(n+\kappa_{j,0})/\lambda_0}\mbf{e}_j,\ i=0.
 \end{eqnarray*}
  The Eichler integrals $\Phi^*_l(\tau)$ are linearly independent. Indeed, otherwise there exist $\beta_1,\cdots, \beta_d$ not all zero such that
 \[\sum_{l=1}^d \beta_l\Phi^*_l(\tau) = 0.\]
 Let $g_\gamma^l(\tau) := (\Phi^*_l|_{-k,\chi,\rho}\gamma)(\tau) - \Phi^*_l(\tau)$. Then one can see that $\sum_{l=1}^d \beta_l<g^l_\gamma> = 0$, where we have written $<g^l_\gamma>$ in place of $<g^l_\gamma|\ \gamma\in\Gamma>$ for simplicity. But this contradicts the assumption that $<g^1_\gamma>, \cdots, <g^d_\gamma>$ are linearly independent.

Using the above, we can show that $\{<g^1_\gamma>, \cdots, <g^d_\gamma>\}$ is a basis of $\tilde{H}^1_{-k,\chi,\rho}(\Gamma,\mc{P})$ as follows.
If $<g_\gamma>\neq0$ is now an element in $\tilde{H}^1_{-k,\chi,\rho}(\Gamma,\mc{P})$, we shall show that $<g_\gamma>$ is a linear combination of $<g^1_\gamma>, \cdots, <g^d_\gamma>$. To this end, as we did above, we first obtain $\Phi(\tau)$ corresponding to $<g_\gamma>$. Once again using Propositon \ref{related2}, we obtain $f(\tau)\in M^!_{-k,\chi,\rho}(\Gamma)$ such that
\begin{eqnarray*}
 \Phi^*(\tau) &=& g_i(\tau) + (\tau-q_i)^k\sum_{j=1}^p \sum_{n+\kappa_{j,i}\geq0}a_i(n,j)e^{\frac{-2\pi i(n+\kappa_{j,i})}{\lambda_i(\tau-q_i)}}\mbf{e}_j,\ 1\leq i\leq t,\\
\Phi^*(\tau) &=& g_0(\tau) + \sum_{l=1}^d a(s_l,t_l)e^{2\pi i(-s_l-\kappa'_{t_l,0})/\lambda_0}\mbf{e}_{t_l} + \sum_{j=1}^p \sum_{n+\kappa_{j,0}\geq0} a_0(n,j)e^{\frac{2\pi i(n+\kappa_{j,0})}{\lambda}}\mbf{e}_j,\ i=0,
 \end{eqnarray*}
where $\Phi^*(\tau) = \Phi(\tau)-f(\tau)$.

Since the Eichler integrals $\Phi^*_l(\tau)$ are linearly independent, we notice that the determinant of the matrix
\begin{equation*}
\sm a^1(s_1,t_1) & \cdot & \cdot & \cdot & a^1(s_d,t_d)\\
    \cdot        & \cdot &       &       & \cdot       \\
    \cdot        &       & \cdot &       & \cdot       \\
    \cdot        &       &       & \cdot & \cdot       \\
    a^d(s_1,t_1) & \cdot & \cdot & \cdot & a^d(s_d,t_d) \esm
\end{equation*}
is non-zero. Otherwise there exist $\beta_1,\cdots, \beta_d$ not all zero such that $\sum_{l=1}^d \beta_l\Phi^*_l(\tau)$ has no poles at infinity. Then $\sum_{l=1}^d \beta_l\Phi^*_l(\tau) \in \mc{P}$. From this we can deduce that $\{\sum_{l=1}^d \beta_l g^l_\gamma|\ \gamma\in\Gamma\}$ is a coboundary, which is a contradiction because $<g^1_\gamma>, \cdots, <g^d_\gamma>$ are linearly independent. Thus, there exist $\alpha_1,\cdots, \alpha_d$ such that
\[a(s_\delta,t_\delta) + \sum_{l=1}^d \alpha_l a^l(s_\delta,t_\delta) = 0,\]
for all $1\leq \delta \leq d$.

Now define
\[\Psi(\tau) := \Phi^*(\tau) + \sum_{l=1}^d \alpha_l\Phi^*_l(\tau).\]
Then $\Psi(\tau)\in \mc{P}$ and
\[(\Psi|_{-k,\chi,\rho}\gamma)(\tau) - \Psi(\tau) = g_\gamma(\tau) + \sum_{l=1}^d \alpha_l g_\gamma^l(\tau).\]
Thus, $\{g_\gamma+ \sum_{l=1}^d\alpha_l g^l_\gamma|\ \gamma\in\Gamma\}$ is a coboundary and hence $<g_\gamma> = -\sum_{l=1}^d \alpha_l<g_\gamma^l>$. This proves that $<g_\gamma>$ is a linear combination of $<g^1_\gamma>, \cdots, <g^d_\gamma>$ and hence $\eta$ is surjective. This completes the proof of Theorem \ref{mainvvmreal}.

\section{Proof of main theorems}  \label{section4}
In this section, we prove the main theorems: Theorem \ref{main1}, Theorem \ref{main2} and Theorem \ref{main3}. To prove Theorem \ref{main1}, we need the following lemma.

\begin{lem} \label{expansion2}
Let $k \in\RR$ and $\chi$ a multiplier system of weight $-k+\frac j2$.
The theta expansion gives an isomorphism
\[H^1_{-k+\frac j2, \mc{M},\chi}(\Gamma^{(1,j)},\mc{P}^e_{\mc{M}})\cong H^1_{-k,\chi'',\rho''}(\Gamma,\mc{P}),\]
and
\[\tilde{H}^1_{-k+\frac j2, \mc{M},\chi}(\Gamma^{(1,j)},\mc{P}^e_{\mc{M}})\cong \tilde{H}^1_{-k,\chi'',\rho''}(\Gamma,\mc{P}),\]
where $\chi''$ and $\rho''$ are given as in (\ref{doublechi}) and (\ref{doublerho}), respectively.
\end{lem}

\begin{proof} [\bf Proof of Lemma \ref{expansion2}]
This proof follows directly from the properties of the theta expansion. Let $\sum_{a\in\mc{N}}f_a(\tau)\mbf{e}_a$ be a vector-valued function on $\HH$. Then we have the following properties
\begin{enumerate}
\item[(1)] for $(\gamma,X)\in\Gamma^{(1,j)}$, we have
\begin{equation*}
\biggl(\biggl(\sum_{a\in\mc{N}}f_a\theta_{2\mc{M},a,0}\biggr)\biggr|_{-k+\frac j2, \mc{M},\chi}(\gamma,X)\biggr)(\tau,z) = \sum_{a\in\mc{N}}F_a(\tau)\theta_{2\mc{M},a,0}(\tau,z),
\end{equation*}
where $\sum_{a\in\mc{N}}F_a(\tau)\mbf{e}_a = \biggl(\biggl(\sum_{a\in\mc{N}}f_a(\tau)\mbf{e}_a\biggr)\biggr|_{-k,\chi'',\rho''}\gamma\biggr)(\tau)$,

\item[(2)] $\sum_{a\in\mc{N}}f_a(\tau)\theta_{2\mc{M},a,0}(\tau,z)\in \mc{P}^e_{\mc{M}}$ if and only if
$\sum_{a\in\mc{N}}f_a(\tau)\mbf{e}_a\in \mc{P}$.
\end{enumerate}
{\color{blue} The first property follows from the transformation formula of the theta function $\theta_{2\mc{M},a,0}(\tau,z)$ and the definition of the representation $\rho''$. The second property comes from the fact that the theta function $\theta_{2\mc{M},a,0}(\tau,z)$ is an element of $\mc{P}_{\mc{M}}^e$.}

For a given cocycle $\{p_{(\gamma,X)}|\ (\gamma,X)\in\Gamma^{(1,j)}\}$ in $\mc{P}^e_{\mc{M}}$, we can write
\begin{equation} \label{cocyclethetaexpansion}
p_{(\gamma,X)}(\tau,z) = \sum_{a\in\mc{N}} g_{a, \gamma}(\tau)\theta_{2\mc{M},a,0}(\tau,z),
\end{equation}
where $g_{\gamma}(\tau) := \sum_{a\in\mc{N}}g_{a,\gamma}(\tau)\mbf{e}_a\in \mc{P}$.
Note that $p_{(\gamma,X)}(\tau,z)$ satisfies the elliptic transformation property as in (\ref{elliptic}) so that $p_{(\gamma,X)}(\tau,z)$ is determined independently of the choice of $X$. Therefore,
$g_{a,\gamma}(\tau)$ is well defined.
 Then, by the properties of theta expansion as we noted above, we see that the collection $\{g_{\gamma}|\ \gamma\in\Gamma\}$ is a cocycle in $\mc{P}$. Conversely, if a cocycle $\{g_{\gamma} = \sum_{a\in\mc{N}}g_{a,\gamma}\mbf{e}_a|\ \gamma\in\Gamma\}$ in $\mc{P}$ is given, then we can construct $p_{(\gamma,X)}(\tau,z)$ as in (\ref{cocyclethetaexpansion}). Then the collection $\{p_{(\gamma,X)}|\ (\gamma,X)\in\Gamma^{(1,j)}\}$ is a cocycle in $\mc{P}^e_{\mc{M}}$. Therefore, there is a bijection between cocycles in $\mc{P}^e_{\mc{M}}$ and cocycles in $\mc{P}$.
 Similarly, we can show that there is a bijection between parabolic cocycles (resp. coboundaries) in  $\mc{P}^e_{\mc{M}}$ and parabolic cocycles (resp. coboundaries) in $\mc{P}$.
 Therefore, the theta expansion induces a map on cocycles (resp. parabolic cocycles) and it sends coboundaries to coboundaries. Hence we have the desired isomorphisms.
 \end{proof}

\begin{proof} [\bf Proof of Theorem \ref{main1}]
First we define the mapping $\tilde{\eta} : S_{k+2+\frac j2,\mc{M},\chi}(\Gamma) \to H^1_{-k+\frac j2, \mc{M}, \chi}(\Gamma^{(1,j)},\mc{P}_{\mc{M}}^e)$.
 Let $\Phi(\tau,z)$ be a Jacobi cusp form in $S_{k+2+\frac j2, \mc{M},\chi}(\Gamma)$. Then $\Phi(\tau,z)$ can be written as follows by the theta expansion
\[\Phi(\tau,z) = \sum_{a\in\mc{N}} f_a(\tau)\theta_{2\mc{M},a,0}(\tau,z).\]
Using this, we define a map $\tilde{\eta}$ as follws: $\tilde{\eta}(\Phi)$ is the class of a cocycle $\{p_{(\gamma,X)}|\ (\gamma,X)\in \Gamma^{(1,j)}\}$, where
\[p_{(\gamma,X)}(\tau,z) =  \sum_{a\in\mc{N}}F_{a,\gamma}(\tau)\theta_{2\mc{M},a,0}(\tau,z)\]
and
\[F_{a,\gamma}(\tau) = \biggl[\int^{i\infty}_{\gamma^{-1}(i\infty)} f_{a}(w)(w-\bar{\tau})^kdw\biggr]^-.\]
In the definition of $\tilde{\eta}$ the map $\tilde{\eta}$ is a composition of three maps:
two maps are induced by the theta expansion and the other map is given by $\eta$ in Theorem \ref{mainvvmreal}.
We know that two maps induced by the theta expansion are isomorphisms by Theorem \ref{isomorphism} and Lemma \ref{expansion2} as follows:
\[S_{k+2+\frac j2,\mc{M},\chi} \cong S_{k+2,\epsilon'',\rho''}(\Gamma)\]
and
\[H^1_{-k+\frac j2,\mc{M},\chi}(\Gamma^{(1,j)},\mc{P}^e_{\mc{M}}) \cong H^1_{-k,\chi'',\rho''}(\Gamma,\mc{P}).\]
Since $\rho''$ is unitary, by Theorem \ref{mainvvmreal} $\eta$ is also an isomorphism. Therefore, $\tilde{\eta}$ is also an isomorphism. This prove Theorem \ref{main1}. By the same way, we can prove Theorem \ref{main2}.
\end{proof}

\begin{proof} [\bf Proof of Theorem \ref{main3}]
Let $H(\tau) := \sum_{\mu=1}^{2m} h_\mu(\tau)\mbf{e}_\mu$ be a vector-valued function on $\HH$, where
\[h_\mu(\tau) = \sum_{n\in\ZZ\atop N(n)>0}C_\mu(N(n))q^{\frac{N(n)}{4m}},\]
where $N(n) = \frac{4m(n+\kappa)}{\lambda}-r^2$.
As in the proof of Theorem \ref{main1}, we see that $\tilde{\eta}(\Phi)$ is a cocycle class given by a cocycle representative $\{r(\Phi, (\gamma,X);\tau,z)|\ (\gamma,X)\in\Gamma^{(1,1)}\}$, where a function
$r(\Phi, (\gamma,X);\tau,z)$ on $\HH\times\CC$ is defined by
\[r(\Phi, (\gamma,X);\tau,z) = \sum_{\mu=1}^{2m} r_{\mu,\gamma}(\tau)\theta_{\mu}(\tau,z).\]
Here, functions $r_{\mu,\gamma}(\tau)$ and $\theta_{\mu}(\tau,z)$ are given by
\[r_{\mu,\gamma}(\tau) = \biggl[\int^{i\infty}_{\gamma^{-1}(i\infty)} h_\mu(w)(w-\bar{\tau})^kdw\biggr]^-,\]
and
\[\theta_{\mu}(\tau,z) = \sum_{r\equiv \mu(\m\ 2m)}q^{r^2}\zeta^r.\]

Now we compute $r_{\mu,\gamma}(\tau)$ explicitly.
Since $\gamma^{-1}(i\infty) = -\frac dc$, we have
\begin{eqnarray*}
r_{\mu,\gamma}(\tau) &=&\biggl[\int^{i\infty}_{-\frac dc} h_\mu(w)(w-\bar{\tau})^k dw\biggr]^-\\
&=&\biggl[\int^{i\infty}_{0} h_\mu\biggl(w-\frac dc\biggr)\biggl(w-\overline{\biggl(\tau+\frac dc\biggr)}\biggr)^k dw\biggr]^-\\
&=& \sum_{n=0}^k \mat k\\ n\emat \biggl(\tau+\frac dc\biggr)^{k-n}(-1)^{k-n}\biggl[\int^{i\infty}_0 h_\mu\biggl(w-\frac dc\biggr)w^ndw\biggr]^-.
\end{eqnarray*}
Then we put into the Fourier expansion of $h_\mu(\tau)$ and we obtain that
\begin{eqnarray*}
\int^{i\infty}_0 h_\mu\biggl(w-\frac dc\biggr)w^ndw &=& i^{n+1}\sum_{n\in\ZZ\atop N(n)>0} C_\mu(N(n))e^{2\pi i\frac{-dN(n)}{4mc}}\int^\infty_0 e^{-2\pi\frac {N(n)}{4m}t}t^n dt\\
&=& i^{n+1}\sum_{n\in\ZZ\atop N(n)>0} C_\mu(N(n))e^{2\pi i\frac{-dN(n)}{4mc}}\biggl(\frac{2\pi N(n)}{4m}\biggr)^{-(n+1)}\Gamma(n+1).
\end{eqnarray*}
Hence, we can check that
\begin{eqnarray*}
r_{\mu,\gamma}(\tau) &=&  \sum_{n=0}^k \mat k\\ n\emat \biggl(\tau+\frac dc\biggr)^{k-n}(-1)^{k-n} \\
&&\times(-i)^{n+1}\sum_{n\in\ZZ\atop N(n)>0} \overline{C_\mu(N(n))}e^{2\pi i\frac{dN(n)}{4mc}}\biggl(\frac{2\pi N(n)}{4m}\biggr)^{-(n+1)}\Gamma(n+1)\\
&=& \sum_{n=0}^{k} \sum_{n\in\ZZ\atop N(n)>0}  \frac{\overline{C_\mu(N(n))} e^{2\pi i\frac{dN(n)}{4mc}}} {\biggl(\frac{N(n)}{4m}\biggr)^{n+1}}\frac{k!(-1)^{k+n}}{(k-n)!(2\pi i)^{n+1}}\biggl(\tau+\frac dc\biggr)^{k-n}\\
&=& \sum_{n=0}^{k} \frac{k!(-1)^{k+n} \overline{L(\Phi, \mu,\gamma,n+1)}}{(k-n)!(2\pi i)^{n+1}}\biggl(\tau+\frac dc\biggr)^{k-n}.
\end{eqnarray*}
In conclusion, we see that
\begin{eqnarray*}
r(\Phi,(\gamma,X);\tau,z) &=& \sum_{\mu=1}^{2m} r_{\mu,\gamma}(\tau)\theta_{\mu}(\tau,z)\\
&=& \sum_{\mu=1}^{2m} \sum_{r\equiv\mu(\m\ 2m)} \sum_{k=0}^{n} \frac{k!(-1)^{k+n} \overline{L(\Phi, \mu,\gamma,n+1)}}{(k-n)!(2\pi i)^{n+1}}\biggl(\tau+\frac dc\biggr)^{k-n}q^{\frac{r^2}{4m}}\zeta^r.
\end{eqnarray*}
This completes the proof.
\end{proof}

\section{Concluding remarks} \label{section5}
The authors of \cite{CL} also defined a coefficient module $\mc{P}_{\mc{M}}$ which is the space of holomorphic functions on $\HH\times\CC^{(j,1)}$ satisfying only the growth condition as in (1) of Definition \ref{dfnofpme}. This space is analogous to the holomorphic function space $\mc{P}$, which was used to establish the Eichler cohomology theory for modular forms of  real weights by Knopp \cite{K}. Therefore, we are led to ask whether an analogue of Theorem \ref{main1} is true for $\mc{P}_{\mc{M}}$ with a suitable replacement for $S_{k+2+\frac j2,\mc{M},\bar{\chi}}(\Gamma^{(1,j)})$.

Comparison of $\mc{P}_{\mc{M}}^e$ and $\mc{P}_{\mc{M}}$ suggests that we have the following exact sequence 
\[0 \to \mc{P}_{\mc{M}}^e \to \mc{P}_{\mc{M}} \to \mc{P}_{\mc{M}}/ \mc{P}_{\mc{M}}^e \to 0.\]
This sequence gives rise to an exact cohomology sequence
\begin{eqnarray*}
0 &\to& H^0_{-k+\frac j2,\mc{M},\chi}(\Gamma^{(1,j)},\mc{P}^e_{\mc{M}}) \to H^0_{-k+\frac j2,\mc{M},\chi}(\Gamma^{(1,j)},\mc{P}_{\mc{M}}) \to H^0_{-k+\frac j2,\mc{M},\chi}(\Gamma^{(1,j)},\mc{P}_{\mc{M}}/ \mc{P}_{\mc{M}}^e) \\
&\to& H^1_{-k+\frac j2,\mc{M},\chi}(\Gamma^{(1,j)},\mc{P}^e_{\mc{M}}) \to H^1_{-k+\frac j2,\mc{M},\chi}(\Gamma^{(1,j)},\mc{P}_{\mc{M}}) \to H^1_{-k+\frac j2,\mc{M},\chi}(\Gamma^{(1,j)},\mc{P}_{\mc{M}}/ \mc{P}_{\mc{M}}^e) \\
&\to& H^2_{-k+\frac j2,\mc{M},\chi}(\Gamma^{(1,j)},\mc{P}^e_{\mc{M}}) \to H^2_{-k+\frac j2,\mc{M},\chi}(\Gamma^{(1,j)},\mc{P}_{\mc{M}}) \to \cdots.
\end{eqnarray*}
Then one can check that
\[H^0_{-k+\frac j2,\mc{M},\chi}(\Gamma^{(1,j)},\mc{P}^e_{\mc{M}}) = H^0_{-k+\frac j2,\mc{M},\chi}(\Gamma^{(1,j)},\mc{P}_{\mc{M}})\] and hence
\[H^0_{-k+\frac j2,\mc{M},\chi}(\Gamma^{(1,j)},\mc{P}_{\mc{M}}/ \mc{P}_{\mc{M}}^e)=0\]
by the exactness of the sequence. From this it follows that the map from $H^1_{-k+\frac j2,\mc{M},\chi}(\Gamma^{(1,j)},\mc{P}^e_{\mc{M}})$ to $H^1_{-k+\frac j2,\mc{M},\chi}(\Gamma^{(1,j)},\mc{P}_{\mc{M}})$ is injective.

In fact, one can check that $H^1_{-k+\frac j2,\mc{M},\chi}(\Gamma^{(1,j)},\mc{P}_{\mc{M}})$ is infinite dimensional. More precisely, let $n\in\RR$ with $n>0$ and $r=(r_1,\cdots, r_j)\in\RR^{(1,j)}$ with $r_1\not\in\ZZ$.  Suppose that $\{p^{n,r}_{(\gamma,X)}|\ (\gamma,X)\in \Gamma^{(1,j)}\}$ is a cocycle of elements of $\mc{P}_{\mc{M}}$ of weight $-k+\frac j2$, index $\mc{M}$ and multiplier system $\chi$ defined by
\begin{eqnarray*}
p^{n,r}_{(\gamma,0)} &=& 0,\ \forall \gamma\in\Gamma,\\
p^{n,r}_{(0,(0,\mbf{e}_i))}&=& 0,\ \forall 1\leq i\leq j,\\
p^{n,r}_{(0,(\mbf{e}_i,0))}&=& 0,\ \forall 2\leq i\leq j,\\
p^{n,r}_{(0,(\mbf{e}_1,0))} &=& q^n e^{2\pi i\tr(rz)},
\end{eqnarray*}
where $\mbf{e}_i$ is the $i$th standard basis of $\RR^{(1,j)}$.
Since $n>0$, a function $q^ne^{2\pi i\tr(rz)}\in \mc{P}_{\mc{M}}$ and a cocycle $\{p^{n,r}_{(\gamma,X)}|\ (\gamma,X)\in \Gamma^{(1,j)}\}$ is well-defined. Then this cocycle is not a coboundary. Otherwise there exists $p(\tau,z)\in \mc{P}_{\mc{M}}$ such that
\[p^{n,r}_{(\gamma,X)}(\tau,z) = (p|_{-k+\frac j2,\mc{M},\chi} (\gamma,X))(\tau,z) - p(\tau,z)\]
for all $(\gamma,X)\in\Gamma^{(1,j)}$. By the definition of $p^{n,r}_{(0,(0,\mbf{e}_1))}(\tau,z)$, we see that
\[(p|_{\mc{M}}(0,\mbf{e}_1))(\tau,z) = p(\tau,z)\]
and hence $p(\tau,z)$ is a period function  with a period $1$  as a function of $z_1$, where $z = (z_1,\cdots, z_j)^t\in \CC^{(j,1)}$. On the other hand, by the definition of $p^{n,r}_{(0,(\mbf{e}_1,0))}(\tau,z)$, we have
\[q^n\zeta^r = (p|_{\mc{M}}(\mbf{e}_1,0))(\tau,z) - p(\tau,z).\]
Taking a slash operator $|_{\mc{M}}(0,\mbf{e}_1)$, one can see that
\begin{eqnarray*}
q^ne^{2\pi i\tr(rz)}e^{2\pi ir_1} &=& (p|_{\mc{M}}{(0,(\mbf{e}_1,0))}{(0,(0,\mbf{e}_1))})(\tau,z) - (p|_{\mc{M}}{(0,(0,\mbf{e}_1))})(\tau,z)\\
&=& (p|_{\mc{M}}(\mbf{e}_1,0))(\tau,z) - p(\tau,z)\\
&=& q^n\zeta^r.
\end{eqnarray*}
But this is a contradiction because $r_1\not\in\ZZ$ and so $e^{2\pi ir_1}\neq1$.
By the same way, we can prove that if $n>0$ and $r_1,\cdots, r_m\not\in\ZZ$  then cocycles $\{p^{n,r_1\mbf{e}_1}_{(\gamma,X)}\}, \cdots \{p^{n,r_m\mbf{e}_1}_{(\gamma,X)}\}$ give linearly independent elements in $H^1_{-k+\frac j2,\mc{M},\chi}(\Gamma^{(1,j)},\mc{P}_{\mc{M}})$, where we have written $\{p^{n,r_i\mbf{e}_1}_{(\gamma,X)}\}$ in place of $\{p^{n,r_i\mbf{e}_1}_{(\gamma,X)}|\ (\gamma,X)\in\Gamma^{(1,j)}\}$ for simplicity. From this, we obtain that the cohomology group $H^1_{-k+\frac j2,\mc{M},\chi}(\Gamma^{(1,j)},\mc{P}_{\mc{M}})$ is infinite dimensional and it is strictly larger than
$H^1_{-k+\frac j2,\mc{M},\chi}(\Gamma^{(1,j)},\mc{P}^e_{\mc{M}})$.

\subsection*{Acknowledgement}
The authors appreciate YoungJu Choie for her several comments.

\bigskip

 
\end{document}